\documentclass[11pt]{amsart}
\usepackage{amsmath,amsfonts,amssymb,upgreek, amscd, mathtools}
\usepackage[breaklinks]{hyperref}

\theoremstyle{thrm}
\usepackage{tikz}
\usepackage{enumerate}
\usetikzlibrary{calc,decorations.markings}
\usepackage{tikz-cd}
\usepackage{graphicx}

\theoremstyle{plain}
\newtheorem{no}{Notation}[section]

\newtheorem{thm}{Theorem}[section]
\newtheorem{lemma}[thm]{Lemma}
\newtheorem{prop}[thm]{Proposition}

\newtheorem{cor}[thm]{Corollary}

\newtheorem{question}[thm]{Question}
\newtheorem{defn}[thm]{Definition}
\newtheorem{example}[thm]{Example}
\newtheorem*{ack}{Acknowledgement}

\theoremstyle{definition}
\newtheorem{remark}[equation]{Remark}

\newcommand{\B}{\operatorname{B} }

\newcommand{\Ho}{\operatorname{H} }

\newcommand{\Z}{\operatorname{Z} }

\newcommand{\im}{\operatorname{im} }

\newcommand{\Fix}{\operatorname{Fix}}

\newcommand{\Id}{\operatorname{id}}

\newcommand{\Aut}{\operatorname{Aut} }

\newcommand{\Ext}{\operatorname{Ext} }

\newcommand{\CExt}{\operatorname{CExt} }

\newcommand{\Ker}{\operatorname{ker} }

\newcommand{\IM}{\operatorname{im} }

\numberwithin{equation}{section}

\begin{document}

\title{Cohomology and Extensions of Relative Rota-Baxter groups}

\author{Pragya Belwal}
\address{Department of Mathematical Sciences, Indian Institute of Science Education and Research (IISER) Mohali, Sector 81, SAS Nagar, P O Manauli, Punjab 140306, India} 
\email{pragyabelwal.math@gmail.com}

\author{Nishant Rathee}
\address{Department of Mathematical Sciences, Indian Institute of Science Education and Research (IISER) Mohali, Sector 81, SAS Nagar, P O Manauli, Punjab 140306, India} 
\email{nishantrathee@iisermohali.ac.in}

\author{Mahender Singh}
\address{Department of Mathematical Sciences, Indian Institute of Science Education and Research (IISER) Mohali, Sector 81, SAS Nagar, P O Manauli, Punjab 140306, India} 
\email{mahender@iisermohali.ac.in} 


\begin{abstract}
Relative Rota-Baxter groups are generalisations of Rota-Baxter groups and recently shown to be intimately related to skew left braces, which are well-known to yield bijective non-degenerate solutions to the Yang-Baxter equation. In this paper, we develop an extension theory of relative Rota-Baxter groups and introduce their low dimensional cohomology groups, which are distinct from the ones known in the context of Rota-Baxter operators on Lie groups. We establish an explicit bijection between the set of equivalence classes of extensions of relative Rota-Baxter groups and their second cohomology. Further, we delve into the connections between this cohomology and the cohomology of associated skew left braces. We prove that for bijective relative Rota-Baxter groups, the two cohomologies are isomorphic in dimension two.
\end{abstract}	

\maketitle

\section{Introduction}
Rota-Baxter operators of weight $1$ on Lie groups were introduced by  Guo, Lang and Sheng in \cite{LHY2021}, with a focus on smooth Rota-Baxter operators so that these operators can be differentiated to yield Rota-Baxter operators of weight $1$ on corresponding Lie algebras. Further study of Rota-Baxter operators on (abstract) groups has been carried out by Bardakov and Gubarev in \cite{VV2022, VV2023}. They showed that every Rota-Baxter operator on a group gives rise to a skew left brace structure on that group, and hence gives a set-theoretical solution to the Yang-Baxter equation. Recall that, a set-theoretical solution to the  Yang Baxter equation is a pair $(X, r)$, where $X$ is a set and $ r: X\times X \longrightarrow X\times X$ is a map (called a braiding) written as $r(x,y) = (\sigma_x(y), \tau_y(x))$ such that the braid equation 
$$(r\times \Id)(\Id \times r)(r\times \Id)=(\Id \times r)(r\times \Id)( \Id \times r)$$
is satisfied. A solution is non degenerate if the maps $x \mapsto \sigma_x$ and $x \mapsto \tau_x$ are bijective for all $x, y \in X$. It is known due to Guarnieri and Vendramin \cite{GV17} that every skew left brace gives rise to a bijective non-degenerate solution to the Yang-Baxter equation. This immediately relates Rota-Baxter operators on groups to the problem of classification of set-theoretical solutions to the  Yang-Baxter equation as envisaged by Drinfel'd \cite{MR1183474}.  It has been shown in \cite{VV2022} that any skew left brace can be embedded into a Rota-Baxter group. A cohomological characterization of skew left braces that can be induced from a Rota-Baxter group has been given in \cite{AS22}, and compelling non-trivial examples of skew left braces that cannot be induced from any Rota-Baxter group have been provided.   The idea of Rota-Baxter operators has been generalised by Jiang, Sheng and Zhu \cite{JYC22} to relative Rota-Baxter operators on Lie groups. Intimate connections of relative Rota-Baxter groups with skew left braces have been explored in \cite{NM1}, where it has been shown that there is an isomorphism between the two categories under some mild conditions.  The recent work \cite{BGST} delves into the study of relative Rota-Baxter groups, their connection to set-theoretical solutions to the  Yang-Baxter equation, Butcher groups, post-groups, and pre-groups. The paper \cite{CMP} offers a definition of Rota-Baxter operators in the framework of Clifford semigroups. Notably, a  connection between Rota-Baxter operators on Clifford semigroups and weak braces has been established, which represent a broader form of skew left braces. Weak braces are  known  to give set-theoretical solutions to the Yang-Baxter equation which need not be non-degenerate \cite{CMMS}.
\par

A cohomology theory for relative Rota-Baxter operators on Lie groups has been introduced in \cite{JYC22}. However, it is worth noting that this cohomology theory is specific to Lie groups since it is defined as the cohomology of the descendent group with coefficients in the Lie algebra of the acting group when viewed as a module over the  descendent group. Further, a classification of abelian extensions of Rota-Baxter groups,  through a suitably defined second cohomology group, has been given in \cite{AN2}. Taking inspiration from the  extensions theory of (abstract) groups and maintaining a focus on connections to skew left braces, we formulate an extension theory tailored for relative Rota-Baxter groups. This involves the introduction of novel cohomology groups in lower dimensions. Moreover, we show that extensions of relative Rota-Baxter groups can be classified through two-dimensional cohomology. We establish connections between extensions of  relative Rota-Baxter groups and that of associated skew left braces and underlying groups.
\par

The paper is organised as follows. After the preliminary Section \ref{section prelim} on the necessary background, we develop the extension theory of relative Rota-Baxter groups in Section \ref{subsec extensions of RRB groups} and  introduce their low dimensional cohomology groups in Section \ref{subsec cohomology RRB groups}. We present one of our main results (Theorem \ref{ext and cohom bijection}) in Section \ref{subsec equivalent extensions and cohomology} which gives a bijection between the set of equivalence classes of extensions of relative Rota-Baxter groups and their second cohomology.  In Section \ref{sec extensions skew braces}, we recall extension theory of skew left braces and show the existence of a homomorphism from the second cohomology of relative Rota-Baxter groups to that of corresponding skew left braces (Proposition \ref{cohom RRB to cohom skew brace}). We prove that for bijective relative Rota-Baxter groups, the two cohomologies are isomorphic in dimension two (Corollary \ref{isomorphism RRB and SLB cohomology}). In Section \ref{sec central and split RRB}, we introduce central and split extensions of relative Rota-Baxter groups, and establish a bijection between semi-direct products and split extensions of such groups (Theorem \ref{semi-direct product and split extensions}).
\medskip

\section{Preliminaries on relative Rota-Baxter groups}\label{section prelim}
	
In this section, we recall some basic notions about relative Rota-Baxter groups that we shall need, and refer the readers to \cite{JYC22, NM1} for more details. We follow the  terminology of \cite{NM1}, which is a bit different from other works in the literature. All through, our groups are written multiplicatively, unless explicitly stated otherwise.

\begin{defn}
A relative Rota-Baxter group is a quadruple $(H, G, \phi, R)$, where $H$ and $G$ are groups, $\phi: G \rightarrow \Aut(H)$ a group homomorphism (where $\phi(g)$ is denoted by $\phi_g$) and $R: H \rightarrow G$ is a map satisfying the condition $$R(h_1) R(h_2)=R(h_1 \phi_{R(h_1)}(h_2))$$ for all $h_1, h_2 \in H$. 
\par
\noindent The map $R$ is referred as the relative Rota-Baxter operator on $H$.
\end{defn}

We say that the relative Rota-Baxter group  $(H, G, \phi, R)$ is trivial if $\phi:G \to \Aut(H)$ is the trivial homomorphism.

\begin{remark}
Note that if $(H, G, \phi, R)$ is  a trivial relative Rota-Baxter group, then $R:H \rightarrow G$ is a group homomorphism. Different homomorphisms give different trivial relative Rota-Baxter groups. Moreover, if $R$ is an isomorphism of groups, then $(H, G, \phi, R)$ is  a trivial relative Rota-Baxter group. 
\end{remark}

\begin{example}
Let $G$ be a group with subgroups $H$ and $L$ such that $G=HL$ and $H\cap L=\{1\}$. Then $(G,G, \phi, R)$ is a relative Rota-Baxter group, where $R: G \rightarrow G$ denotes the map given by $R(hl) = l^{-1}$ and $\phi : G \rightarrow \Aut(G)$ is the adjoint action, that is, $\phi_g(x) = gxg^{-1}$ for $g, h \in G$.
\end{example}

\begin{example}
Let $G$ be a group, and $G^{\rm op}$ its opposite group. Let $\phi: G^{\rm op} \rightarrow \Aut(G)$ be the homomorphism given by $\phi_x(y) = x^{-1} y x$ for all $x, y \in G$. Then the quadruple $(G, G^{\rm op}, \phi, \Id_G)$ is a relative Rota-Baxter group.
\end{example}

\begin{example}
Take $H= \mathbb{R}$ and $G=UP(2; \mathbb{R})$, the group of invertible upper triangular matrices. Let $\phi:UP(2; \mathbb{R})\to \Aut(\mathbb{R})$ be given by

$$\phi_{\begin{pmatrix} a & b \\ 0 & c \end{pmatrix}}(r)= ar$$
for $\begin{pmatrix} a & b \\ 0 & c \end{pmatrix} \in UP(2; \mathbb{R})$ and $r \in \mathbb{R}$. Further, let $R: \mathbb{R} \to  UP(2; \mathbb{R})$ be given by $$R(r)=\begin{pmatrix} 1 & r \\ 0 & 1 \end{pmatrix}.$$ Then $(\mathbb{R},UP(2; \mathbb{R}), \phi, R)$ is a relative Rota-Baxter group.
\end{example}

	Let  $(H, G, \phi, R)$ be a relative Rota-Baxter group, and let $K \leq H$ and $L \leq G$ be subgroups.
\begin{enumerate}
	\item If $K$ is $L$-invariant under the action $\phi$, then we denote the restriction of $\phi$ by $\phi|: L \to \Aut(K)$. 
	\item If $R(K) \subseteq L$, then we denote the restriction of $R$ by $R|: K \to L$.
\end{enumerate}

\begin{defn}
	Let $(H,G,\phi,R)$ be a relative Rota-Baxter group, and $K\leq H$ and $L\leq G$ be subgroups. Suppose that  $\phi_\ell(K) \subseteq K$ for all $\ell \in L$ and $R(K) \subseteq L$. Then $(K,L,\phi |,R |)$ is a relative Rota-Baxter group, which we refer as a relative Rota-Baxter subgroup of $(H,G,\phi,R)$ and write $(K,L,\phi |,R |)\leq(H,G,\phi,R)$.
\end{defn}

\begin{defn}\label{defn ideal rbb-datum}
	Let $(H, G, \phi, R)$ be a relative Rota-Baxter group and  $(K, L,  \phi|, R|) \leq (H, G, \phi, R)$ its relative Rota-Baxter subgroup. We say that $(K, L,  \phi|, R|)$ is an ideal of $(H, G, \phi, R)$ if 
	\begin{align}
		& K \trianglelefteq H \quad \mbox{and} \quad L \trianglelefteq G, \label{I0}\\
		& \phi_g(K) \subseteq K  \mbox{ for all } g \in G, \label{I1} \\
		& \phi_\ell(h) h^{-1} \in K \mbox{ for all } h \in H \mbox{ and }  \ell \in L. \label{I2}
	\end{align}
	We write $(K, L, \phi|, R|) \trianglelefteq (H, G, \phi, R)$ to denote an ideal of a relative Rota-Baxter group. 
\end{defn}

The preceding definitions lead to the following result \cite[Theorem 5.3]{NM1}.

\begin{thm}\label{subs}
	Let $(H, G, \phi, R)$ be a relative Rota-Baxter group and $(K, L,  \phi|, R|)$ an ideal of $(H, G, \phi, R)$. Then there are maps $\overline{\phi}: G/L \to \Aut(H/K)$ and  $\overline{R}: H/K \to G/L$ defined by
	$$ \overline{\phi}_{\overline{g}}(\overline{h})=\overline{\phi_{g}(h)} \quad \textrm{and} \quad	\overline{R}(\overline{h})=\overline{R(h)}$$
	for $\overline{g} \in G/L$ and $\overline{h} \in H/K$, such that  $(H/K, G/L, \overline{\phi}, \overline{R})$ is a relative Rota-Baxter group.
\end{thm}
\begin{no}
	We write $(H, G, \phi, R)/(K, L, \phi|, R|)$ to denote the quotient relative Rota-Baxter group $(H/K, G/L, \overline{\phi}, \overline{R})$.
\end{no}

\begin{defn}
	Let $(H, G, \phi, R)$ and $(K, L, \varphi, S)$ be two relative Rota-Baxter groups.
	\begin{enumerate}
		\item A homomorphism $(\psi, \eta): (H, G, \phi, R) \to (K, L, \varphi, S)$ of relative Rota-Baxter groups is a pair $(\psi, \eta)$, where $\psi: H \rightarrow K$ and $\eta: G \rightarrow L$ are group homomorphisms such that
		\begin{equation}\label{rbb datum morphism}
			\eta \; R = S \; \psi \quad \textrm{and} \quad \psi \; \phi_g =  \varphi_{\eta(g)}\; \psi
		\end{equation}
for all $g \in G$.
		\item The kernel of a homomorphism $(\psi, \eta): (H, G, \phi, R) \to (K, L, \varphi, S)$ of relative Rota-Baxter groups is the quadruple $$(\Ker(\psi), \Ker(\eta), \phi|, R|),$$ where $\Ker(\psi)$ and $\Ker(\eta)$ denote the kernels of the group homomorphisms $\psi$ and $\eta$, respectively. The conditions in \eqref{rbb datum morphism} imply that the kernel is itself a relative Rota-Baxter group. In fact, the kernel turns out to be an ideal of $(H, G, \phi, R)$.
		
		\item The image of a homomorphism $(\psi, \eta):  (H, G, \phi, R) \to (K, L, \varphi, S)$ of relative Rota-Baxter groups is the quadruple 
		$$(\IM(\psi), \IM(\eta), \varphi|, S| ),$$ where $\IM(\psi)$ and $\IM(\eta)$ denote the images of the group homomorphisms $\psi$ and $\eta$, respectively. The image is itself a relative Rota-Baxter group.
		
		\item A homomorphism $(\psi, \eta)$ of relative Rota-Baxter groups is called an isomorphism if both $\psi$ and $\eta$ are group isomorphisms. Similarly, we say that $(\psi, \eta)$ is an embedding of a relative Rota-Baxter group if both $\psi$ and $\eta$ are embeddings of groups.
	\end{enumerate}
\end{defn}

\begin{defn}
A Rota-Baxter group is a group $G$ together with a map $R: G \rightarrow G$ such that
$$ R(x)  R(y)= R(x  R(x)  y  R(x)^{-1}) $$
for all $x, y \in G$. The map $R$ is referred as the Rota-Baxter operator on $G$.
\end{defn}
Let $\phi : G \rightarrow \Aut(G)$ be the adjoint action, that is, $\phi_g(x)=gxg^{-1}$ for $g, h \in G$. Then the relative Rota-Baxter group $(G, G, \phi, R)$ is simply a Rota-Baxter group.

\begin{prop}\label{R homo H to G} \cite[Proposition 3.5]{JYC22} 
Let $(H, G, \phi, R)$ be a relative Rota-Baxter group. Then the operation 
\begin{align}
h_1 \circ_R h_2 = h_1 \phi_{R(h_1)}(h_2)
\end{align}
defines a group operation on $H$.  Moreover, the map $R: H^{(\circ_R)} \rightarrow G$ is a group homomorphism. The group $H^{(\circ_R)}$ is called the descendent group of $R$.
\end{prop}

If  $(H, G, \phi, R)$ is a relative Rota-Baxter group, then the image $R(H)$ of $H$ under $R$ is  a subgroup of $G$.

\begin{defn}
	The center of a relative Rota-Baxter group $(H, G, \phi, R)$ is defined as
	$$\Z(H, G, \phi, R)= \big( \Z^{\phi}_R(H), \Ker(\phi), \phi|, R| \big),$$
	where $\Z^{\phi}_{R}(H)=\Z(H)  \cap \Ker(\phi \,R) \cap \Fix(\phi)$, $\Fix(\phi)= \{ x \in H  \mid \phi_g(x)=x \mbox{ for all } g \in G \}$ and  $R: H^{(\circ_R)} \rightarrow G$ is viewed as a group homomorphism.
\end{defn}

Let us recall the definition of a skew left brace.

\begin{defn}
	A skew left brace is a  triple $(H,\cdot ,\circ)$, where $(H,\cdot)$ and $(H, \circ)$ are groups  such that
	$$a \circ (b \cdot c)=(a\circ b) \cdot a^{-1} \cdot (a \circ c)$$
	holds for all $a,b,c \in H$, where $a^{-1}$ denotes the inverse of $a$ in $(H, \cdot)$. The groups $(H,\cdot)$ and $(H, \circ)$ are called the additive and the multiplicative groups of the skew left brace $(H, \cdot, \circ)$, and will sometimes be abbreviated as $H ^{(\cdot)}$ and  $H ^{(\circ)}$, respectively.
\end{defn}

A skew left brace $(H, \cdot, \circ)$  is said to be  trivial  if $a \cdot b= a \circ b$ for all $a, b \in H$.

\begin{prop}\label{rrb2sb}\cite[Proposition 3.5]{NM1} 
	Let $(H, G, \phi, R)$ be a relative Rota-Baxter group. If $\cdot$ denotes the  group operation of $H$, then the triple $(H, \cdot, \circ_R)$ is a skew left brace. 
\end{prop}

If $(H, G, \phi, R)$ is a relative Rota-Baxter group, then $(H, \cdot, \circ_R)$ is referred as the skew left brace induced by $R$ and will be denoted by $H_R$ for brevity. The following indispensable result is immediate \cite[Proposition 4.3]{NM1}.

\begin{prop}\label{rrb to slb homo}
A homomorphism of relative Rota-Baxter groups induces a homomorphism of induced skew left braces. 
\end{prop}
\medskip

\section{Extensions and cohomology of relative Rota-Baxter groups}\label{sec cohomology and extensions of RRB groups}

Let $A$ be a group and $K$ be an abelian group.  Then $K$ is said to be a right $A$-module if there exists an anti-homomorphism $\mu: A \rightarrow \Aut(K)$. Writing $k \cdot a= \mu_a(k)$, we see that 
$$k \cdot (ab)= \mu_{ab}(k)=  \mu_{b} \mu_{a}(k)= (k \cdot a) \cdot b$$ 
for all $a, b \in A$ and $k \in K$.

Let $K$ be a right $A$-module via an anti-homomorphism $\mu: A \to \Aut(K)$. For each $n \geq 1$, let $C^n(A, K)$ be the group of all maps $A^n \to K$ which vanish on degenerate tuples, where a tuple $(a_1,\ldots, a_n) \in A^n$ is called degenerate if $a_i = 1$ for at least one $i$. Further, let $\partial_\mu^n: C^n(A,K) \rightarrow C^{n+1}(A,K)$ be defined by
\begin{eqnarray}\label{group coboundary mu}
&& (\partial_\mu^n f)(a_1, a_2,\ldots, a_{n+1}) \\
&= & f(a_2, \ldots,a_{n+1})~ \prod^{n}_{k=1}    (f(a_1, \ldots, a_k \cdot a_{k+1}, \ldots, a_{n+1}))^{(-1)^k}~ \mu_{a_{n+1}}(f(a_1, \ldots,a_{n})), \notag
\end{eqnarray} 
where $\cdot$ is the group operation in $A$. Then, $\{C^n(A, K)$, $\partial_\mu^n\}$ forms a cochain complex. 
\par 

Now, let $(A,B, \phi, R)$ be a relative Rota-Baxter group and $L$ a right $B$-module via an anti-homomorphism $\sigma: B \to \Aut(L)$. Let $\{C^n(B, L)$, $\partial^{n}_{\sigma}\}$ be the corresponding cochain complex. We have a homomorphism $R: A^{(\circ_R)} \rightarrow B$, which allows us to turn $L$  into a right $A^{(\circ_R)}$-module through the composition of $\sigma$ and $R$.  Consequently, we can define the corresponding cochain complex $\{C^n(A^{(\circ_R)}, L), \delta_{\sigma}^n\}$, where $\delta_{\sigma}^n$ is defined by 
\begin{eqnarray}\label{group coboundary twisted sigma}
&&	(\delta_{\sigma}^n f)(a_1, a_2,\ldots, a_{n+1})\\
 &= & f(a_2, \ldots,a_{n+1}) ~\prod^{n}_{k=1}    (f(a_1, \ldots, a_k \circ_R a_{k+1}, \ldots, a_{n+1}))^{(-1)^k} ~ \sigma_{R(a_{n+1})}(f(a_1, \ldots,a_{n})). \notag
\end{eqnarray}
\medskip

\subsection{Extensions of relative Rota-Baxter groups}\label{subsec extensions of RRB groups}
Here and henceforth, we write $\bf{1}$ to denote the trivial relative Rota-Baxter group for which both the underlying groups and the maps are trivial.

\begin{defn}
Let $(K,L, \alpha,S )$ and $(A,B, \beta, T)$ be relative Rota-Baxter groups.  
\begin{enumerate}
\item An extension of $(A,B, \beta, T)$ by $(K,L, \alpha,S )$ is a relative Rota-Baxter group  $(H,G, \phi, R)$ that fits into the sequence 
$$\mathcal{E} : \quad  {\bf 1} \longrightarrow (K,L, \alpha,S ) \stackrel{(i_1, i_2)}{\longrightarrow}  (H,G, \phi, R) \stackrel{(\pi_1, \pi_2)}{\longrightarrow} (A,B, \beta, T) \longrightarrow {\bf 1},$$ 
where  $(i_1, i_2)$ and $(\pi_1, \pi_2)$ are morphisms of relative Rota-Baxter groups such that $(i_1, i_2)$ is an embedding, $(\pi_1, \pi_2)$ is an epimorphism of  relative Rota-Baxter groups and $(\IM(i_1), \IM(i_2), \phi|, R|)= (\Ker(\pi_1), \Ker(\pi_2), \phi|, R|)$.
\item[] To avoid complexity of notation, we assume that $i_1$ and $i_2$ are inclusion maps. This allows us to write that $\phi$ restricted to  $L$ is $\alpha$ and $R$ restricted to  $K$ is $S$.	
		\item We say that $\mathcal{E}$ is an abelian extension if $K$ and $L$ are abelian groups and the relative Rota-Baxter group $(K,L, \alpha,S )$ is trivial.
		
		\item We say that $\mathcal{E}$ is a central extension if $K \leq \Z^{\phi}_{R}(H)$ and $L \leq \Z(G) \cap \Ker(\phi)$. Clearly, every central extension is abelian.
	\end{enumerate}
\end{defn}

By abuse of notation, let ${\bf 1}$ also denote the trivial skew left brace for which both the underlying groups are trivial.

\begin{defn}
Let $Q$ and $I$ be skew left braces. An extension of $Q$ by $I$ is a skew left brace $E$ that fits into the sequence
$$\mathcal{E} : \quad  {\bf 1}\longrightarrow I \stackrel{i}{\longrightarrow}  E \stackrel{\pi}{\longrightarrow} Q \longrightarrow {\bf 1},$$ 
where $i$ and $\pi$ are morphisms of skew left braces such that $\im(i)=\ker(\pi)$, $i$ is injective and $\pi$ is surjective.
\end{defn}

\begin{remark}\label{extsb}
By Proposition \ref{rrb to slb homo}, a morphism of relative Rota-Baxter groups induces a morphism of induced skew left braces. Thus, it follows that an extension 
$$\mathcal{E}: \quad {\bf 1} \longrightarrow (K,L, \alpha,S ) \stackrel{(i_1, i_2)}{\longrightarrow}  (H,G, \phi, R) \stackrel{(\pi_1, \pi_2)}{\longrightarrow} (A,B, \beta, T) \longrightarrow {\bf 1}$$ 
of relative Rota-Baxter groups induces an extension
$$\mathcal{E}_{SB}: \quad  {\bf 1} \longrightarrow K_{S} \stackrel{i_1}{\longrightarrow}  H_{R} \stackrel{\pi_1}{\longrightarrow} A_{T} \longrightarrow {\bf 1}$$
of induced skew left braces. We shall see later in Corollary \ref{RRB2sbext} that equivalent extensions of relative Rota-Baxter groups map to equivalent extensions of skew left braces.
\end{remark}

\begin{remark}\label{remark1}
An abelian extension of relative Rota-Baxter groups induces an extension of a skew left brace by  a trivial brace. Such extensions of skew left braces have been explored in \cite{DB18, NMY1}. Similarly, a central extension of relative Rota-Baxter groups induces a central extension of skew left braces, which have been investigated in \cite{CCS3, LV16}.
\end{remark}

Let $(H,\cdot ,\circ)$ be a skew left brace. Then the map $\lambda^H: H^{(\circ)} \longrightarrow \operatorname{Aut}(H^{(\cdot)})$ defined by 
$$\lambda^H_a(b) = a^{-1} \cdot (a \circ b),$$ 
for $a,b \in H$, is a group homomorphism. It turns out that the identity map $\Id_H: H^{(\cdot)} \longrightarrow H^{(\circ)}$ is a relative Rota-Baxter operator with respect to the action $\lambda^H$, and hence  the quadruple $(H^{(\cdot)}, H^{(\circ)}, \lambda^H, \Id_H)$ is a relative Rota-Baxter group corresponding to the skew left brace $(H,\cdot ,\circ)$.

\begin{prop} \label{rrb induced by slb}\cite[Proposition 3.8]{NM1}
The skew left brace  $(H,\cdot ,\circ)$ is induced by the relative Rota-Baxter group  $(H^{(\cdot)}, H^{(\circ)}, \lambda^H, \Id_H)$.
\end{prop}

\begin{prop}\label{slb to rrb homo} \cite[Proposition 4.4]{NM1}
A homomorphism of skew left braces induces a homomorphism of corresponding relative Rota-Baxter groups.
\end{prop}

In view of the preceding discussion, we may inquire whether every extension of skew left braces arise from an extension of relative Rota-Baxter groups.

\begin{prop}\label{skew brace extension from rb extension}
Every extension of skew left braces is induced by an extension of relative Rota-Baxter groups.
\end{prop}

\begin{proof}
Consider the extension
$$ \mathcal{E}: \quad {\bf 1} \longrightarrow K \stackrel{i}{\longrightarrow}  H \stackrel{\pi}{\longrightarrow} A \longrightarrow {\bf 1}$$
 of skew left braces. Then we have  relative Rota-Baxter groups $(K^{(\cdot)}, K^{(\circ)}, \lambda^K, \Id_K)$, $(H^{(\cdot)}, H^{(\circ)}, \lambda^H, \Id_H)$ and $(A^{(\cdot)}, A^{(\circ)}, \lambda^A, \Id_A)$. Since $i$ and $\pi$ are morphisms of skew left braces, using Proposition \ref{slb to rrb homo}, we obtain the following extension of relative Rota-Baxter groups
$$\mathcal{E}_{RRB} : \quad  {\bf 1} \longrightarrow (K^{(\cdot)}, K^{(\circ)}, \lambda^K, \Id_K) \stackrel{(i, i)}{\longrightarrow}  (H^{(\cdot)}, H^{(\circ)}, \lambda^H, \Id_H) \stackrel{(\pi, \pi)}{\longrightarrow} (A^{(\cdot)}, A^{(\circ)}, \lambda^A, \Id_A) \longrightarrow {\bf 1}.$$
It is immediate that the extension of skew left  braces induced by $\mathcal{E}_{RRB}$  is identical to $\mathcal{E}$, which completes the proof.
\end{proof}

\begin{prop}
An extension
$$\mathcal{E} : \quad {\bf 1} \longrightarrow (K,L, \alpha,S ) \stackrel{(i_1, i_2)}{\longrightarrow}  (H,G, \phi, R) \stackrel{(\pi_1, \pi_2)}{\longrightarrow} (A,B, \beta, T) \longrightarrow {\bf 1}$$
of relative Rota-Baxter groups	induces extensions of groups $$\mathcal{E}_1 : \quad  1 \longrightarrow K \stackrel{i_1}{\longrightarrow}  H \stackrel{\pi_1 }{\longrightarrow} A \longrightarrow 1 \quad \textrm{and} \quad \mathcal{E}_2 : \quad 1 \longrightarrow L \stackrel{i_2}{\longrightarrow}  G \stackrel{\pi_2 }{\longrightarrow} B \longrightarrow 1.$$
	Furthermore, $ (K,L, \alpha,S )$ is an ideal of $(H,G, \phi, R)$ and the quotient relative Rota-Baxter group $(H, G, \phi, R)/ (K,L, \alpha,S ) $ is isomorphic to $(A, B, \beta, T).$
\end{prop}

\begin{proof}
The first assertion is immediate. Since $(K,L, \alpha,S)$ is the kernel of the homomorphism $(\pi_1, \pi_2)$ of relative Rota-Baxter groups, it follows that it is an ideal of $(H,G, \phi, R)$. Let $\bar{\pi}_1: H/K \rightarrow A$ and $\bar{\pi}_2: G/L \rightarrow B$ be induced isomorphisms of groups given by $\bar{\pi}_1(\bar{h})=\pi_1(h)$ and $\bar{\pi}_2(\bar{g})=\pi_2(g)$ for $h \in H$ and $g \in G$. Further, since $\bar{\pi}_2 \bar{R}= T\bar{\pi}_1$ and $\bar{\pi}_1 \bar{\phi}_{\bar{g}}= \beta_{\bar{\pi}_2(\bar{g})} \bar{\pi}_1$ for all $\bar{g} \in G/L$, it follows that $(\bar{\pi}_1, \bar{\pi}_2): (H/K, G/L, \bar{\phi}, \bar{R}) \to (A, B, \beta, T)$ is an isomorphism of relative Rota-Baxter groups. 
\end{proof}

Given an extension
$$\mathcal{E} : \quad  {\bf 1} \longrightarrow (K,L, \alpha,S ) \stackrel{(i_1, i_2)}{\longrightarrow}  (H,G, \phi, R) \stackrel{(\pi_1, \pi_2)}{\longrightarrow} (A,B, \beta, T) \longrightarrow {\bf 1}$$
 of  relative Rota-Baxter groups, we obtained the following  extensions of groups
$$\mathcal{E}_1 : \quad  1 \longrightarrow K \stackrel{i_1}{\longrightarrow}  H \stackrel{\pi_1 }{\longrightarrow} A \longrightarrow 1 \quad \textrm{and} \quad \mathcal{E}_2 : \quad  1 \longrightarrow L \stackrel{i_2}{\longrightarrow}  G \stackrel{\pi_2 }{\longrightarrow} B \longrightarrow 1.$$

Since $(\pi_1, \pi_2)$ is a homomorphism of relative Rota-Baxter groups, we have 
\begin{equation}\label{compat pi1 pi2}
\pi_2 \,R= T \, \pi_1 \quad \textrm{and} \quad \pi_1\, \phi_g= \beta_{\pi_2(g)} \,\pi_1
\end{equation}
for all $g \in G$. Let $s_{H}$ and $s_{G}$ be normalised set-theoretic sections to $\mathcal{E}_1 $ and $\mathcal{E}_2$, respectively. The ordered pair $(s_H, s_G)$ will be referred as a set-theoretic section to the extension $\mathcal{E}$. Then each element $h \in H$ and $g \in G$ can be uniquely written as  $h= s_H(a)\; k$ and $g=s_G(b) \; l$ for some $k \in K$, $l \in L$ and $a \in A$, $b \in B$. We have 
\begin{eqnarray}\label{leqn}
	\phi_g(h) &= & \phi_{ s_G(b) \; l}( s_H(a) \; k) \notag \\
&= & \phi_{ s_G(b)}( \phi_l(s_H(a)) \; \phi_l (k))  \notag  \\
	&= &   \phi_{ s_G(b)}( s_H(a)) \;   \phi_{s_G(b)} (f(l,a) \alpha_l(k)),
\end{eqnarray}
where the map $f: L \times A \longrightarrow K$ is defined by
\begin{equation}\label{fdefn}
 f(l,a)= s_H( a)^{-1} \phi_l(s_H(a)) .
\end{equation}
 Further, we can write 
\begin{equation}\label{atilde1}
	\phi_{s_G(b)}(s_H(a)) =   s_H(\tilde{a}) \;\rho(a, b) 
\end{equation}
for unique elements  $\rho(a, b) \in K$ and $\tilde{a} \in A$. This gives a map $\rho: A \times B \longrightarrow K$. Now applying $\pi_1$ on both the sides of \eqref{atilde1} and using \eqref{compat pi1 pi2}, we get 
\begin{equation}\label{va}
	\beta{_b(a)}=  \tilde{a}.
\end{equation}
Using \eqref{va} in \eqref{atilde1}, we can write
\begin{equation}\label{atilde}
	\phi_{s_G(b)}(s_H(a)) =   s_H(\beta{_b(a)}) \; \rho(a, b).
\end{equation}
Using \eqref{atilde} in \eqref{leqn}, we have 
\begin{equation}\label{feqn}
	\phi_g(h)=   s_H(\beta{_b(a)})~ \rho(a,b)  ~ \phi_{s_G(b)} (f(l,a)\alpha_l(k)) .
\end{equation}
\par 

Next, we see how a relative Rota-Baxter operator can be defined on an extension of relative Rota-Baxter groups  $(K,L, \alpha,S )$ and $(A,B, \beta, T)$. If $h=s_H(a) k$ is an element of $H$, then we have  
\begin{equation}\label{RRB1}
 R( s_H(a) k  ) 	= R(s_H(a) \phi_{R(s_H(a))}(\phi^{-1}_{R(s_H(a))}(k)) )=  R(s_H(a)) R(\phi^{-1}_{R(s_H(a))}(k)).
\end{equation}

Let $R(s_H(a))=s_G(b)  \chi(a) $ for unique elements $b \in B$ and $\chi(a) \in L$. This gives a map $\chi:A \longrightarrow L$. Applying $\pi_2$ on both the sides of \eqref{RRB1} and using \eqref{compat pi1 pi2}, we get $T(a)=b$, and hence
\begin{equation}\label{RRB2}
R(s_H(a)) = s_G(T(a))  ~ \chi(a).
\end{equation}
Putting  value of $R(s_H(a))$ from \eqref{RRB2} in \eqref{RRB1}, we get

\begin{equation}\label{relativecon}
	R( s_H(a)   k  )  =   s_G(T(a))  \chi(a) R(\phi^{-1}_{ s_G(T(a)) \;  \chi(a)}( k)).
\end{equation}

\begin{prop}\label{construction of actions}
Consider the abelian extension 
$$\mathcal{E} : \quad  {\bf 1} \longrightarrow (K,L, \alpha,S ) \stackrel{(i_1, i_2)}{\longrightarrow}  (H,G, \phi, R) \stackrel{(\pi_1, \pi_2)}{\longrightarrow} (A,B, \beta, T) \longrightarrow {\bf 1}$$
of relative Rota-Baxter groups. Then the following hold:
\begin{enumerate}
\item  The map  $\nu: B \rightarrow \Aut(K)$ defined by 
\begin{equation}\label{nuact}
\nu_b(k)= \phi_{s_G(b)}(k)
\end{equation}
 for $b \in B$ and $k \in K$, is a homomorphism of groups.
 \item  The $\mu: A \rightarrow \Aut(K)$ defined by 
\begin{equation}\label{muact}
\mu_a(k)=s_H(a)^{-1}\, k\, s_H(a)
\end{equation}
 for $a \in A$ and $ k \in K$, is an anti-homomorphism of groups.
\item The map $\sigma: B \rightarrow \Aut(L)$ defined by 
\begin{equation}\label{sigmaact}
\sigma_b(l)=s_G(b)^{-1}\,l\, s_G(b)
\end{equation}
 for $b \in B$ and $ l \in K$, is an anti-homomorphism of groups.
\end{enumerate}
Further, all the maps are independent of the choice of a section to $\mathcal{E}$.
\end{prop}

\begin{proof}
If $b_1, b_2 \in B$, then $s_G(b_1b_2) \tau_2(b_1, b_2)= s_G(b_1) s_G(b_2) $ for some $ \tau_2(b_1, b_2) \in L$. This gives
$$\nu_{b_1b_2}(k)= \phi_{s_G(b_1b_2)}(k)= \phi_{s_G(b_1) s_G(b_2)  \tau_2(b_1, b_2)^{-1}}(k)=\phi_{s_G(b_1)} \phi_{s_G(b_2)} \phi_{ \tau_2(b_1, b_2)^{-1}}(k)$$
and
$$\nu_{b_1}\nu_{b_2}(k)= \phi_{s_G(b_1)} \phi_{s_G(b_2)}(k).$$
Since $\phi_{ \tau_2(b_1, b_2)^{-1}}=\alpha_{\tau_2(b_1, b_2)^{-1}}=\Id_K$, we have $\phi_{ \tau_2(b_1, b_2)^{-1}}(k)=k$, and hence $\nu_{b_1b_2}= \nu_{b_1}\nu_{b_2}$. If $t_G$ is another section, then we have $s_G(b)=t_G(b) \ell$ for some $\ell \in L$. Since  $\phi_\ell(k)=k$, we have $\phi_{s_G(b)}(k)= \phi_{t_G(b) \ell}(k)=\phi_{t_G(b)} \phi_{\ell}(k)=\phi_{t_G(b)}(k)$, and hence $\nu$ is independent of the choice of the section. This proves assertion (1).
\par

If $a_1, a_2 \in A$, then $s_H(a_1a_2) \tau_1(a_1, a_2)= s_H(a_1) s_H(a_2)$ for some $ \tau_1(a_1, a_2) \in K$. This gives
\begin{eqnarray*}
\mu_{a_1a_2}(k) &=& s_H(a_1a_2)^{-1} ~k ~s_H(a_1a_2)\\
 &=&   \tau_1(a_1, a_2) s_H(a_2)^{-1}s_H(a_1)^{-1} ~k~s_H(a_1) s_H(a_2) \tau_1(a_1, a_2)^{-1}\\
 &=&   \tau_1(a_1, a_2) s_H(a_2)^{-1} \, \mu_{a_1}(k) \,s_H(a_2) \tau_1(a_1, a_2)^{-1}\\
  &=&   \tau_1(a_1, a_2) \, \mu_{a_2}\mu_{a_1}(k) \, \tau_1(a_1, a_2)^{-1}\\
    &=&   \mu_{a_2}\mu_{a_1}(k), \quad \textrm{since}~ K~ \textrm{is abelian}.
\end{eqnarray*}
Again, since $K$ is abelian, it follows that $\mu$ is independent of the choice of the section $s_H$, which proves (2). Proof of assertion (3) is analogous.
\end{proof}

Henceforth, we assume that all extensions of relative Rota-Baxter groups are abelian. Let  $$\mathcal{E} : \quad {\bf 1} \longrightarrow (K,L, \alpha,S ) \stackrel{(i_1, i_2)}{\longrightarrow}  (H,G, \phi, R) \stackrel{(\pi_1, \pi_2)}{\longrightarrow} (A,B, \beta, T) \longrightarrow {\bf 1}$$
be an abelian extension of relative Rota-Baxter groups and $(s_H, s_G)$ a set-theoretic section to $\mathcal{E}$. The classical extension theory of groups gives the following (see \cite{MR0672956}):

\begin{enumerate}
\item The map $\tau_1: A \times A \rightarrow K$ given by
\begin{equation}\label{mucocycle}
\tau_1(a_1, a_2)= s_H(a_1 a_2)^{-1}s_H(a_1)s_H(a_2)
\end{equation}
for $a_1, a_2 \in A$ is a group 2-cocycle with respect to the action $\mu$.
\item The map $\tau_2: B \times B \rightarrow L$ given by
\begin{equation}\label{sigmacocycle}
	\tau_2(b_1, b_2)= s_G(b_1 b_2)^{-1}s_G(b_1)s_G(b_2)
\end{equation}
$b_1, b_2 \in B$ is a group 2-cocycle with respect to the action $\sigma$.
\end{enumerate}

Given the abelian extension $\mathcal{E}$, using extension theory of groups, we can take $H= A \times_{\tau_1} K $  with the group operation given by 
$$(a_1,k_1) (a_2, k_2)=\big(a_1 a_2,~ \mu_{a_2}(k_1)\, k_2\, \tau_1(a_1, a_2) \big).$$
Similarly, we can take $G=B \times_{\tau_2} L$ with the group operation given by
$$(b_1,l_1) (b_2, l_2)=\big(b_1 b_2, ~\sigma_{b_2}(l_1)\,l_2\, \tau_2(b_1, b_2)\big).$$

It follows from \eqref{feqn} and \eqref{relativecon} that the maps $\phi: B \times_{\tau_2} L \longrightarrow \Aut(A \times_{\tau_1} K )$ and  $R:  A \times_{\tau_1} K \longrightarrow B \times_{\tau_2} L$ are given by
\begin{eqnarray*}
\phi_{(b,l)}(a,k) &= & \big(\beta{_{b}(a)}, ~\rho(a,b) \,\nu_b(f(l,a)k)\big),\\
R(a,k) &= & \big(T(a), ~\chi(a) \, S(\nu^{-1}_{T(a)}(k))\big)
\end{eqnarray*} 
for $(a,k) \in A \times_{\tau_1} K$ and $(b,l) \in B \times_{\tau_2} L$. Now, we determine conditions on $\rho$ and $\nu$ for $\phi$ to be a group homomorphism. 

Let $(b_1,l_1), (b_2,l_2) \in B \times_{\tau_2} L$ and $(a,k) \in A \times_{\tau_1} K$. Then we have 
\begin{eqnarray}\label{iso1}
&& \phi_{(b_1,l_1) (b_2, l_2)}(a,k)\\
&=& \phi_{(b_1 b_2, \sigma_{b_2}(l_1)l_2\tau_2(b_1, b_2))}(a,k)\notag \\
&=& \big ( \beta{_{b_1 b_2}(a)},~ \rho(a, b_1 b_2) \,\nu_{b_1 b_2}( f(\sigma_{b_2}(l_1)l_2\tau_2(b_1, b_2), a)k)\big) \notag
\end{eqnarray} 

and 
\begin{eqnarray}\label{iso2}
&& \phi_{(b_1,l_1)}(\phi_{(b_2,l_2)}(a,k))\\
&= & \phi_{(b_1,l_1)}\big( \beta{_{b_2}(a)}, ~\rho(a,b_2) \,\nu_{b_2}(f ( l_2,a)k)\big) \notag \\
&=& \big(\beta{_{b_1 b_2}(a)},~\rho(\beta{_{b_2}(a)}, b_1)\, \nu_{b_1}\big(f(l_1, \beta{_{b_2}(a))}\, \rho(a,b_2)\, \nu_{b_2}(f ( l_2,a)k) \big)\big). \notag
\end{eqnarray}

Comparing \eqref{iso1} and \eqref{iso2} and using the fact that $\beta$ and $\nu$ are homomorphisms, we get
\begin{eqnarray}\label{iso5}
&& \rho(a, b_1 b_2) \,\nu_{b_1 b_2}\big( f(\sigma_{b_2}(l_1)l_2\tau_2(b_1, b_2), a)\big)\\
&= & \rho(\beta{_{b_2}(a)}, b_1) \, \nu_{b_1}\big (f(l_1, \beta{_{b_2}(a))}\rho(a,b_2)  \nu_{b_2}(f ( l_2,a))\big) \notag.
\end{eqnarray}

Taking $l_1, l_2 =1$ in \eqref{iso5} gives
\begin{eqnarray}\label{iso3}
\rho(a, b_1 b_2) \,\nu_{b_1 b_2}\big( f(\tau_2(b_1, b_2), a)\big) &=& \rho(\beta{_{b_2}(a)}, b_1) \,\nu_{b_1}(\rho(a,b_2)),
\end{eqnarray} 
whereas taking $b_1, b_2=1$ yields
\begin{eqnarray}\label{iso4}
	f(l_1l_2,a) &=& f( l_1, a) \,f( l_2,a).
\end{eqnarray}

Using \refeq{iso3} and \eqref{iso4} in \eqref{iso5}, we get 
\begin{eqnarray}
	\nu_{b_2}(f(\sigma_{b_2}(l_1), a)) &=& f( l_1, \beta{_{b_2}(a))}.
\end{eqnarray}
Next, we determine conditions  under which $\phi_{(b,l)}$ is homomorphism for all $(b,l) \in  B \times_{\tau_2} L$. Let $(a_1,k_1), (a_2, k_2) \in A \times_{\tau_1} K$. Then, we have
\begin{eqnarray}\label{hom1}
&& \phi_{(b,l)} \big((a_1,k_1) (a_2, k_2)\big) \\
&=& \phi_{(b,l)}(a_1 a_2, ~\mu_{a_2}(k_1)\,k_2 \,\tau_1(a_1, a_2))\notag\\
&=& \big(\beta{_{b}(a_1 a_2)}, ~\rho(a_1 a_2,b)  \,\nu_b \big( f(l, a_1 a_2)\mu_{a_2}(k_1)k_2\tau_1(a_1, a_2)\big)\big) \notag
\end{eqnarray}
and
\begin{eqnarray}\label{hom2}
	&& \phi_{(b,l)}(a_1,k_1)\, \phi_{(b,l)}(a_2,k_2)\\
	&= & \big(\beta{_{b}(a_1)}, ~\rho(a_1,b)  \, \nu_{b} (f(l,a_1)k_1)\big) \,\big(\beta{_{b}(a_2)}, ~ \rho(a_2,b)  \,\nu_b(f(l,a_2) k_2)\big)\notag \\
	&= & \big(\beta{_{b}(a_1 a_2)}, ~\mu_{ \beta{_{b}(a_2)}} \big(\rho(a_1,b)   \nu_{b} (f(l,a_1)k_1)\big)\, \rho(a_2,b)  \nu_b(f(l,a_2) k_2) ~\tau_1(\beta{_{b}(a_1)}, \beta_{b}(a_2)) \big). \notag
\end{eqnarray}

Comparing  \eqref{hom1} and \eqref{hom2} give
\begin{eqnarray}\label{mainhom}
&&\rho(a_1 a_2,b)~  \nu_b \big( f(l, a_1 a_2)\mu_{a_2}(k_1)k_2\tau_1(a_1, a_2) \big) \\
&=& \mu_{ \beta{_{b}(a_2)}} \big(\rho(a_1,b)   \nu_{b} (f(l,a_1)k_1)\big) \, \rho(a_2,b)  \nu_b(f(l,a_2) k_2)  \,\tau_1(\beta{_{b}(a_1)}, \beta{_{b}(a_2))}. \notag
\end{eqnarray}
Taking $b=1$ in \eqref{mainhom} gives
\begin{eqnarray}\label{fcon}
	f(l, a_1 a_2) &=& \mu_{ a_2}(f(l,a_1))f(l, a_2),
\end{eqnarray}
whereas taking $a_1=1$ gives
\begin{eqnarray}\label{mcon}
	\nu_b(\mu_{a_2}(k_1)) &=&   \mu_{ \beta{_{b}(a_2)}}( \nu_{b} (k_1)).
\end{eqnarray}
Using \eqref{fcon} and \eqref{mcon} in \eqref{mainhom}, we obtain
\begin{eqnarray}\label{rho second equation}
	\rho(a_1 a_2, b) \,\nu_{b}(\tau_1(a_1, a_2)) &=& \mu_{ \beta{_{b}(a_2)}}(\rho(a_1,b)) \,\rho(a_2, b) \, \tau_1(\beta{_{b}(a_1)}, \beta{_{b}(a_2)}).
\end{eqnarray}

Putting succinctly, we have the following result.

	\begin{prop}\label{phiallcondn}
		The map $\phi: B \times_{\tau_2} L \longrightarrow \Aut(A \times_{\tau_1} K )$ defined by
		\begin{eqnarray}\label{phidefnprop}
			\phi_{(b,l)}(a,k) &= & \big(\beta{_{b}(a)}, ~\rho(a,b)\nu_b(f(l,a)k)\big),
		\end{eqnarray} 
		for $(a,k) \in A \times_{\tau_1} K$ and $(b,l) \in B \times_{\tau_2} L$, is a homomorphism if and only if the conditions 
		\begin{eqnarray*}
			\nu_b(\mu_{a_2}(k)) & = &   \mu_{ \beta{_{b}(a_2)}}( \nu_{b} (k)),\\
			\rho(a_1, b_1 b_2) \,\nu_{b_1 b_2}( f(\tau_2(b_1, b_2), a_1)) &=& \rho(\beta{_{b_2}(a_1)}, b_1) \,\nu_{b_1}(\rho(a,b_2)),\\
			f(l_1l_2,a_1) &=& f( l_1, a_1)f( l_2,a_1),\\
			f(l, a_1 a_2) &=& \mu_{ a_2}(f(l,a_1))f(l, a_2),\\
			\nu_{b_2}(f(\sigma_{b_2}(l_1), a_1)) &=& f( l_1, \beta{_{b_2}(a_1))},\\
			\rho(a_1 a_2, b_1) \,\nu_{b_1}(\tau_1(a_1, a_2)) &=& \mu_{ \beta{_{b_1}(a_2)}}(\rho(a_1,b_1)) \,\rho(a_2, b_1) \, \tau_1(\beta{_{b_1}(a_1)}, \beta{_{b_1}(a_2)}),
		\end{eqnarray*} 
hold for all $a_1, a_2 \in A$, $b_1, b_2 \in B$ and $k \in K$.
\end{prop}

\begin{lemma}\label{properties of f}
Let $\mathcal{E}$ be an abelian extension of relative Rota-Baxter groups. Then the following hold:
\begin{enumerate}
\item The map $f: L \times A \longrightarrow K$ defined by
$$ f(l,a)= s_H( a)^{-1} \phi_l(s_H(a))$$
for $l \in L$ and $a \in A$, is independent of the choice of the section $s_H$.
\item $f(l_1l_2,a) = f( l_1, a)f( l_2,a)$ for all $l_1, l_2 \in L$ and $a \in A$.
\item $f(l, a_1 a_2) = \mu_{ a_2}(f(l,a_1))f(l, a_2)$ for all $l \in L$ and $a_1, a_2 \in A$.
\end{enumerate}
\end{lemma}

\begin{proof}
Let $s_{H}$ and $t_H$ be two set-theoretic sections of $\mathcal{E}_1$. For each $a \in A$, there exists a unique element $\lambda(a) \in K$ such that $s_{H}(a)= t_H(a) \lambda(a)$. Note that $(K, L , \alpha, S)$ is an ideal of $(H, G , \phi, R)$ and is also a trivial relative Rota-Baxter group. This gives
\begin{eqnarray*}
f(l,a) &=& s_H( a)^{-1} \phi_l(s_H(a))\\
 &=&  \lambda(a)^{-1} t_H(a)^{-1} \phi_l(t_H(a)\lambda(a))\\
  &=&  \lambda(a)^{-1} t_H(a)^{-1} \phi_l(t_H(a)) \phi_l(\lambda(a))\\
  &=&  \lambda(a)^{-1} t_H(a)^{-1} \phi_l(t_H(a)) \lambda(a), \quad \textrm{since}~ \phi_l(\lambda(a))=\lambda(a)\\
  &=&   t_H(a)^{-1} \phi_l(t_H(a)), \quad \textrm{since}~ t_H(a)^{-1} \phi_l(t_H(a)) \in K.
\end{eqnarray*}
Thus, $f$ is independent of the choice of the section, which proves assertion (1). Assertions (2) and (3) are simply \eqref{iso4} and \eqref{fcon}, respectively.
\end{proof}

Our next aim is to determine conditions  under which $R$ becomes a relative Rota-Baxter operator. Let $(a_1,k_1)$ and $(a_2, k_2) \in A \times_{\tau_1} K$. Recall that $a_1 \circ_{T} a_2=a_1 \beta{}_{T(a_1)}(a_2)$. Then we have
\begin{eqnarray}\label{RRBcon}
	&& R\big( (a_1,k_1) \phi_{R(a_1,k_1)} (a_2, k_2) \big)\\
	&=& R \big( (a_1, k_1) \phi_{(T(a_1), ~\chi(a_1) S(\nu^{-1}_{T(a_1)}(k_1)))}(a_2, k_2) \big)\notag\\
	&=& R \big( (a_1, k_1) \big(\beta{_{T(a_1)}}(a_2), ~\rho(a_2, T(a_1)) \,\nu_{T(a_1)} \big(f(\chi(a_1) S(\nu^{-1}_{T(a_1)}(k_1)), a_2)  k_2 \big) \big)\big)\notag\\
	&=& R \big( a_1 \circ_{T} a_2,~ \mu_{\beta{_{T(a_1)}(a_2)}}(k_1) \, \rho(a_2, T(a_1)) \,\nu_{T(a_1)} \big(f(\chi(a_1) S(\nu^{-1}_{T(a_1)}(k_1)), a_2)k_2 \big) \,\tau_1(a_1, \beta{_{T(a_1)}(a_2)}) \big)\notag\\
	&=& \big(T(a_1 \circ_{T} a_2), ~\chi(a_1 \circ_{T}  a_2) \,S(\nu^{-1}_{T(a_1 \circ_{T} a_2)} (z)) \big), \text{where} \notag
\end{eqnarray}
\begin{eqnarray*}
	z&=&\; \mu_{\beta{_{T(a_1)}(a_2)}}(k_1)\, \rho(a_2, T(a_1)) \,\nu_{T(a_1)} \big(f(\chi(a_1) S(\nu^{-1}_{T(a_1)}(k_1)), a_2) k_2 \big)  \, \tau_1(a_1,\beta{_{T(a_1)}(a_2)}).
\end{eqnarray*}

On the other hand, we have
\begin{eqnarray}\label{RRBcon2}
&& R (a_1,k_1)R (a_2,k_2)\\
&= & \big(T(a_1),~ \chi(a_1) S(\nu^{-1}_{T(a_1)}(k_1)) \big)\, \big(T(a_2), ~\chi(a_2) S(\nu^{-1}_{T(a_2)}(k_2))\big)\notag\\
&=& \big(T(a_1 \circ_{T} a_2),~ \sigma_{T(a_2)}\big( \chi(a_1) S(\nu^{-1}_{T(a_1)}(k_1))\big) \, \chi(a_2) \,S(\nu^{-1}_{T(a_2)}(k_2)) \,\tau_2(T(a_1), T(a_2) ) \big) \notag.
\end{eqnarray} 

Comparing \eqref{RRBcon} and \eqref{RRBcon2} and using the fact that $T(a_1 \circ_{T} a_2)=  T(a_1)T(a_2)$, we obtain
\begin{eqnarray}\label{Rmain}
	&& \chi(a_1 \circ_{T} a_2) \,S\big( \nu^{-1}_{T(a_1 \circ_{T} a_2)} \big(\mu_{\beta{_{T(a_1)}}(a_2)}(k_1) \, \rho(a_2, T(a_1))\, \nu_{T(a_1)}(f(\chi(a_1) S(\nu^{-1}_{T(a_1)}(k_1)), a_2)k_2) \\
	&&  \tau_1(a_1,\beta{_{T(a_1)}}(a_2))\big) \big)\notag\\
	&=& \sigma_{T(a_2)}\big( \chi(a_1) S(\nu^{-1}_{T(a_1)}(k_1))\big)\, \chi(a_2)\, \tau_2 \,(T(a_1), T(a_2) )S(\nu^{-1}_{T(a_2)}(k_2)) . \notag
\end{eqnarray}

Cancelling off $S(\nu^{-1}_{T(a_2)}(k_2))$ and putting $k_1 = 1$ in \eqref{Rmain}, we get 
\begin{small}
\begin{eqnarray}\label{R1}
&& \chi(a_1 \circ_{T} a_2) \,S \big(\nu^{-1}_{T(a_1 \circ_{T} a_2)}\big(\rho(a_2, T(a_1)) \,\tau_1(a_1,\beta{_{T(a_1)}}(a_2))\big) \,\nu^{-1}_{T(a_2)}(f(\chi(a_1), a_2))\big) \\
&= &  \sigma_{T(a_2)}( \chi(a_1)) \, \chi(a_2) \,\tau_2(T(a_1), T(a_2) ). \notag
\end{eqnarray}
\end{small}
Using \eqref{R1} in \eqref{Rmain} and the facts that $f$ is linear in the first coordinate and  $T(a_1 \circ_{T} a_2)=  T(a_1)T(a_2)$, we get
\begin{equation}
S\big(\nu^{-1}_{T(a_1 \circ_{T} a_2)}(\mu_{\beta{_{T(a_1)}}(a_2)}(k_1))\nu^{-1}_{T(a_2)}(f(S(\nu^{-1}_{T(a_1)}(k_1)), a_2))\big)= \sigma_{T(a_2)}(S(\nu^{-1}_{T(a_1)}(k_1))).
\end{equation}
Replacing $k_1$ by $\nu_{T(a_1)}(k_1)$ and using \eqref{mcon}, we get
\begin{eqnarray}\label{finalactrel}
	S\big(\nu^{-1}_{T(a_2)}(\mu_{a_2}(k_1)) \,\nu^{-1}_{T( a_2)}(f(S(k_1), a_2))\big) &=& \sigma_{T(a_2)}(S(k_1)).
\end{eqnarray}

The quadruple $(\nu, \mu, \sigma, f)$ is called the \emph{associated action} of the extension $\mathcal{E}$ of relative Rota-Baxter groups. The preceding observations lead to the following result.

\begin{prop}\label{Rallcondn}
Let $\phi: B \times_{\tau_2} L \longrightarrow \Aut(A \times_{\tau_1} K )$ defined in \eqref{phidefnprop} be a homomorphism. Then the map  $R:  A \times_{\tau_1} K \rightarrow B \times_{\tau_2} L$ given by
\begin{eqnarray*}
R(a,k) &= & \big(T(a), ~\chi(a) \,S(\nu^{-1}_{T(a)}(k))\big),
\end{eqnarray*} 
for $(a,k) \in A \times_{\tau_1} K$ and $(b,l) \in B \times_{\tau_2} L$, is a relative Rota-Baxter operator if and only if the conditions 
\begin{eqnarray*}
& & \chi(a_1 \circ_{T} a_2) \,S \big(\nu^{-1}_{T(a_1 \circ_{T} a_2)}\big(\rho(a_2, T(a_1)) \,\tau_1(a_1,\beta{_{T(a_1)}}(a_2))\big) \,\nu^{-1}_{T(a_2)}(f(\chi(a_1), a_2))\big) \\
& =&   \sigma_{T(a_2)}( \chi(a_1)) \, \chi(a_2) \,\tau_2(T(a_1), T(a_2) )\\
\textrm{and} &&\\
&& S\big(\nu^{-1}_{T(a_2)}(\mu_{a_2}(k)) \,\nu^{-1}_{T( a_2)}(f(S(k), a_2))\big) = \sigma_{T(a_2)}(S(k)),	
\end{eqnarray*}		
hold for all $ a_1, a_2 \in A$ and $ k \in K$.
\end{prop}

\begin{remark}\label{bijext}
It follows from Proposition \ref{Rallcondn} that if $T$ and $S$ are bijections, then $R$ is also a bijection. Consequently, every extension of bijective relative Rota-Baxter groups is itself bijective.
\end{remark}

A special case of the preceding discussion gives the direct product of relative Rota-Baxter groups. 

\begin{example}
Let $\mathcal{K}=(K,L,\alpha,S) $ and $\mathcal{A}=(A,B,\beta,T)$ be two relative Rota-Baxter groups. Then the direct product of $\mathcal{K}$ and $\mathcal{A}$ is defined as $$\mathcal{A}\times\mathcal{K}=(A\times K, \, B\times L, \, \beta\times \alpha, \, T\times S),$$ where $A\times K$ and $B\times L$ are direct product of groups,  and the action and the Rota-Baxter operator are given by
	\begin{eqnarray*}
		(\beta\times\alpha)_{(b,l)}(a,k)&=&(\beta_b(a),\,\alpha_l(k)),\\
		(T\times S)(a,k)&=&(T(a),\,S(k)),
	\end{eqnarray*}
for $a \in A$, $b \in B$, $k \in K$ and $l \in L$.
\end{example}
\medskip

\subsection{Cohomology of relative Rota-Baxter groups}\label{subsec cohomology RRB groups}
Based on the relationship between $\mu, \sigma, \nu, f$ and their properties derived in the preceding subsection, we now define a module over a relative Rota-Baxter group.

\begin{defn}\label{moddefn}
	A module over a relative Rota-Baxter group $(A, B,\beta, T)$ is a trivial relative Rota-Baxter group $(K, L, \alpha, S)$ such that there exists a quadruple $(\nu, \mu, \sigma, f)$ (called action) of maps satisfying the following conditions:
	\begin{enumerate}
		\item The group $K$ is a left $B$-module and a right $A$-module with respect to the actions $\nu:B \to \Aut(K) $ and $\mu: A \to \Aut(K)$, respectively.
		\item The group $L$ is a right $B$-module with respect to the action $\sigma: B \to \Aut(L)$.
		\item The  map  $f: L \times A \to K$ has the property that $f(-,a): L \rightarrow K$ is a homomorphism for all $a \in A$ and $f(l,-): A \rightarrow K$ is a derivation with respect to the action $\mu$ for all $l \in L$.
		\item $S\big(\nu^{-1}_{T(a)}(\mu_{a}(k)) \,\nu^{-1}_{T( a)}(f(S(k), a))\big)=\;\sigma_{T(a)}(S(k))$ for all $ a \in A$ and $k \in K$.
		\item $\nu_b(\mu_{a}(k))= \;   \mu_{ \beta{_{b}(a)}}( \nu_{b} (k))$ for all $a \in A$, $b \in B$ and $k \in K$. 
	\end{enumerate}
\end{defn}

We say that the module $(K, L, \alpha, S)$ is trivial if all the maps $\mu, \nu, \sigma, f$ are trivial. 

\begin{prop}\label{module structure from extension}
	Suppose that $${\bf 1} \longrightarrow (K,L, \alpha,S ) \stackrel{(i_1, i_2)}{\longrightarrow}  (H,G, \phi, R) \stackrel{(\pi_1, \pi_2)}{\longrightarrow} (A,B, \beta, T) \longrightarrow {\bf 1}$$
	is an abelian extension of relative Rota-Baxter groups. Then  the following holds:
	\begin{enumerate}
		\item The group $K$ is a left $B$-module via an action $\nu$  and a right $A$-module via an action $\mu$, where $\nu$ and $\mu$ are defined in \eqref{nuact} and  \eqref{muact}, respectively.
		\item The group $L$ is a right $B$-module via an action $\sigma$, where $\sigma$ is defined in \eqref{sigmaact}.
		\item  The trivial relative Rota-Baxter group $(K, L, \alpha, S)$ is an $(A,B, \beta, T)$-module via the action $(\nu, \mu, \sigma, f)$.
	\end{enumerate}
\end{prop}
\begin{proof}
	Assertions (1) and (2) follow from Proposition \ref{construction of actions}. The existence of the map $f: L \times A \to K$ is given by \eqref{fdefn}. That $f$ is linear in the first variable and is a derivation in the second variable follow from \eqref{iso4} and \eqref{fcon}, respectively. The two desired relations among these maps follow from \eqref{mcon} and  \eqref{finalactrel}, which proves assertion (3).
\end{proof}

Using the relationships between the maps $\tau_1, \tau_2, f,\rho,\chi$ obtained in the preceding subsection,  we now define low dimensional cohomology groups of a  relative Rota-Baxter group. Let $\mathcal{A}=(A, B, \beta, T)$ and $\mathcal{K}=(K, L, \alpha, S)$ be  relative Rota-Baxter groups such that   $\mathcal{K}$ is trivial and is an $\mathcal{A}$-module via the action $(\nu, \mu, \sigma, f)$.
\medskip 

Let us set 
\begin{eqnarray*}
	C^0_{RRB} & = & K \times_{(\nu, \mu, \sigma, f) } L,\\
	C^{1}_{RRB} &=& C^1(A, K) \oplus C^1(B,L),\\
	C^{2}_{RRB} &=& C^2(A, K) \oplus C^2(B,L) \oplus C(A \times B, K) \oplus C(A,L),\\
	C^{3}_{RRB} &=& C^3(A, K) \oplus C^3(B,L) \oplus C(A \times B^2, K) \oplus C(A^2 \times B, K) \oplus C^2(A,L),
\end{eqnarray*}
where 
\begin{eqnarray*}
	K \times_{(\nu, \mu, \sigma, f) } L &=& \big\{ (k,l) \in K \times L ~\mid ~ f(\sigma_b(l), a)=f(l,a),~ \nu_b(k) = k,\\
	&& \sigma_{T(a)}(l)=l ~\textrm{and}~ S(\mu_a(k))=S(k) ~\textrm{for all}~ a \in A~\textrm{and}~ b \in B \big\}
\end{eqnarray*} 
and  $C(A^m \times B^n, K)$ is the group of maps which vanish on degenerate tuples, i.e, on tuples in which either $a_i=1$ for some $1\leq i\leq m$ or  $b_j=1$ for some $1\leq j\leq n$.
\par

Define $\partial_{RRB}^0: C^0_{RRB} \rightarrow C^{1}_{RRB}$ by 
$$\partial_{RRB}^0(k,l)=(\kappa_k, \,\omega_l),$$
where $\kappa_k: A \rightarrow K$ and $\omega_l: B \rightarrow L$ are defined by 
\begin{eqnarray}
	\kappa_k(a) &=&\mu_a(k)k^{-1},\\
	\omega_l(b)&=& \sigma_b(l)l^{-1},
\end{eqnarray} 
for $k \in K$, $l \in L$, $a \in A$ and $b \in B$. Define $\partial_{RRB}^1: C^1_{RRB} \rightarrow C^{2}_{RRB}$ by 
$$\partial_{RRB}^1(\theta_1, \theta_2)=(\partial_{\mu}^1(\theta_1), \partial_{\sigma}^1(\theta_2), \lambda_1, \lambda_2  ),$$ 
where  $\partial_{\mu}^1,  \partial_{\sigma}^1$ are as in \eqref{group coboundary mu} and  $\lambda_1, \lambda_2$ are given by
\begin{eqnarray*}
	\lambda_1(a,b) &=& \nu_b \big(f(\theta_2(b), a)\theta_1(a) \big)\, \big(\theta_1(\beta_b(a))\big)^{-1},\\
	\lambda_2(a) &=& S \big(\nu^{-1}_{T(a)}(\theta_1(a)) \big) \, \big(\theta_2(T(a)) \big)^{-1},
\end{eqnarray*}
for $\theta_1 \in C^1(A, K)$, $\theta_2 \in C^1(B,L)$, $a \in A$ and $b \in B$. Finally, define $\partial_{RRB}^2: C^2_{RRB} \rightarrow C^{3}_{RRB}$  by  
$$\partial^2_{RRB}(\tau_1, \tau_2, \rho, \chi)=(\partial_{\mu}^2(\tau_1), \partial_{\sigma}^2(\tau_2), \gamma_1, \gamma_2, \gamma_3),$$
where  $\partial_{\mu}^2,  \partial_{\sigma}^2$ are as in \eqref{group coboundary mu} and
\begin{eqnarray*}
	\gamma_1(a,b_1, b_2) &=& \rho(a, b_1 b_2) \, \nu_{b_1 b_2}( f(\tau_2(b_1, b_2), a)) \, (\rho(\beta_{b_2}(a), b_1))^{-1}  \,(\nu_{b_1}(\rho(a,b_2)))^{-1},\\
	\gamma_2(a_1, a_2,b) &=& \rho(a_1 a_2, b) \, \nu_{b}(\tau_1(a_1, a_2))  \,(\mu_{ \beta_{b}(a_2)}(\rho(a_1,b)))^{-1}  \,(\rho(a_2, b))^{-1}  \,(\tau_1(\beta_{b}(a_1), \beta_{b}(a_2)))^{-1},\\
	\gamma_3(a_1, a_2) &= &  S \big(\nu^{-1}_{T(a_1 \circ_T a_2)}\big(\rho(a_2, T(a_1))  \, \tau_1(a_1,\beta_{T(a_1)}(a_2))  \,\nu_{T(a_1)}(f(\chi(a_1), a_2))  \big)\big)\\
	&& (\tau_2(T(a_1), T(a_2) ))^{-1}  \,(\delta^1_{\sigma}(\chi)(a_1, a_2))^{-1}
\end{eqnarray*}
for $\tau_1  \in C^2(A, K)$,  $\tau_2 \in C^2(B,L)$, $\rho \in C(A \times B, K)$, $\chi \in C(A,L)$, $a, a_1, a_2 \in A$ and $b, b_1, b_2 \in B$.

\begin{lemma}\label{rrb coboundary condition}
	$\IM(\partial^0_{RRB}) \subseteq  \Ker(\partial^1_{RRB})$ and	$\IM(\partial^1_{RRB}) \subseteq  \Ker(\partial^2_{RRB})$.
\end{lemma}

\begin{proof}
It is enough to prove that $\partial^1_{RRB}\;  \partial^0_{RRB}$ and $\partial^2_{RRB}\;  \partial^1_{RRB}$ are trivial homomorphisms. We have 
	\begin{eqnarray*}
		\partial^1_{RRB}\;  \partial^0_{RRB}(k,l) &=& (\partial^1_{\mu}(\kappa_k), \partial^1_{\sigma}(\omega_l), \lambda_1, \lambda_2).
	\end{eqnarray*}
	Since $\kappa_k$ and $\omega_l$ are group 1-coboundaries,  we have $\partial^1_{\mu}(\kappa_k)=1$ and $\partial^1_{\sigma}(\omega_l)=1$. Further, we have
	\begin{eqnarray*}
		\lambda_1(a,b) &=& \nu_b \big(f(\sigma_b(l)l^{-1}, a)\mu_a(k)k^{-1} \big)  \,(\mu_{\beta_b(a)}(k))^{-1} \,k.
	\end{eqnarray*}
	Using the definition of $K \times_{(\nu, \mu, \sigma, f)} L$ and the compatibility conditions in Definition \ref{moddefn}, we have $\lambda_1(a, b) = 1$. Similarly, we obtain $\lambda_2(a, b) = 1$ for all $a \in A$ and $b \in B.$ 
		\par 
	
	Now we prove that $\partial^2_{RRB}\;  \partial^1_{RRB}$ is trivial. If $(\theta_1, \theta_2) \in C^1_{RRB}$, then
	\begin{eqnarray}\label{com1}
		\partial^2_{RRB}\;  \partial^1_{RRB}(\theta_1, \theta_2) &=& \partial^2_{RRB}(\partial_{\mu}^1(\theta_1), \partial_{\sigma}^1(\theta_2), \lambda_1, \lambda_2 ).
	\end{eqnarray}
	We show that each term on the right hand side of \eqref{com1} is a trivial map. The first two terms are trivial maps since $\partial_{\mu}^2\partial_{\mu}^1$ and $  \partial_{\sigma}^2\partial_{\sigma}^1$ are trivial being the coboundary maps defining the usual group cohomology. We expand the remaining three terms using the definitions of $\partial^1_{RRB}$ and $\partial^2_{RRB}$. It follows from \eqref{group coboundary mu} and \eqref{group coboundary twisted sigma} that
	\begin{eqnarray*}
		\partial^1_{\mu}(\theta_1)(a_1, a_2) &=& \theta_1(a_2)\, (\theta_1(a_1 a_2))^{-1} \,\mu_{a_2}(\theta_1(a_1)),\\
		\partial^1_{\sigma}(\theta_2)(b_1, b_2) &=& \theta_2(b_2) \,(\theta_2(b_1 b_2))^{-1} \,\sigma_{b_2}(\theta_2(b_1)),\\
		\delta^1_{\sigma}(\chi)(a_1, a_2) &=& \chi(a_2) \, (\chi(a_1 \circ_T a_2))^{-1} \,\sigma_{T(a_2)}( \chi(a_1)).
	\end{eqnarray*}
	
	The third term of \eqref{com1} is given by
	\begin{eqnarray}\label{t1}
		&& \gamma_1(a, b_1, b_2) \notag\\
		&=&	  \nu_{b_1 b_2}(f(\theta_2(b_1 b_2), a)\theta_1(a))(\theta_1(\beta_{b_1 b_2}(a)))^{-1} \nu_{b_1 b_2}(f( \theta_2(b_2)(\theta_2(b_1 b_2))^{-1}\sigma_{b_2}(\theta_2(b_1)), a))\notag \\ 
		&&(\nu_{b_1}(f(\theta_2(b_1), \beta_{b_2}(a))\theta_1(\beta_{b_2}(a)))(\theta_1(\beta_{b_1}(\beta_{b_2}(a))))^{-1})^{-1}(\nu_{b_1}\big( \nu_{b_2}(f(\theta_2(b_2), a)\theta_1(a))(\theta_1(\beta_{b_2}(a)))^{-1})\big)^{-1}. \notag
	\end{eqnarray}
Using that $f(\theta_2(b_1), \beta_{b_2}(a))= \nu_{b_2}(f(\sigma_{b_2}(\theta_2(b_2)),a))$ and $f$ is linear in the first coordinate, we have $\gamma_1(a, b_1, b_2)=1$. The fourth term of \eqref{com1} is given by
	\begin{eqnarray}
		\gamma_2(a_1, a_2,b) &= & \nu_b(f(\theta_2(b), a_1a_2)\theta_1(a_1 a_2))(\theta_1(\beta_b(a_1 a_2)))^{-1} \nu_{b}(\theta_1(a_2)(\theta_1(a_1 a_2))^{-1}\mu_{a_2}(\theta_1(a_1))) \notag\\
		&&  (\mu_{ \beta_{b}(a_2)}(\nu_b(f(\theta_2(b), a_1)\theta_1(a_1))(\theta_1(\beta_b(a_1)))^{-1}))^{-1} (\nu_b(f(\theta_2(b), a_2)\theta_1(a_2)) (\theta_1(\beta_b(a_2)))^{-1})^{-1}\notag\\
		&&(\theta_1(\beta_{b}(a_2)))^{-1}\theta_1(\beta_{b}(a_1)\beta_{b}(a_2))(\mu_{\beta_{b}(a_2)}(\theta_1(\beta_{b}(a_1))))^{-1}.\notag
	\end{eqnarray}
	Using that $f(\theta_2(b), a_1 a_2) =  \mu_{ a_2}(f(\theta_2(b),a_1))f(\theta_2(b), a_2)$ and $\nu_b(\mu_{ a_2}(k))=\mu_{ \beta_{b}(a_2)}(\nu_b(k))$ for all $k \in K$, we have $\gamma_2(a_1, a_2,b)=1$. The fifth term is given by 
	\begin{eqnarray*}
		\gamma_3(a_1, a_2) &= & \chi(a_1 \circ_T a_2) ~S(\nu^{-1}_{T(a_1 \circ a_2)}\big(\rho(a_2, T(a_1))\tau_1(a_1,\beta_{T(a_1)}(a_2)) \nu_{T(a_1)}(f(\chi(a_1), a_2)))\big)\notag\\
		&& (\sigma_{T(a_2)}( \chi(a_1)))^{-1} ~(\chi(a_2))^{-1} ~(\tau_2(T(a_1), T(a_2) ))^{-1} \notag\\
		&= &  S(\nu^{-1}_{T(a_1 \circ_T a_2)}(\theta_1(a_1 \circ_T a_2))) (\theta_2(T(a_1 \circ_T a_2)))^{-1} S(\nu^{-1}_{T(a_1 \circ_T a_2)}\big( \nu_{T(a_1)}(f(\theta_2(T(a_1)), a_2)\notag\\
		&& \theta_1(a_2))(\theta_1(\beta_{T(a_1)}(a_2)))^{-1} \theta_1(\beta_{T(a_1)}(a_2))(\theta_1(a_1 \beta_{T(a_1)}(a_2)))^{-1}\mu_{\beta_{T(a_1)}(a_2)}(\theta_1(a_1))\notag\\
		&&  \nu_{T(a_1)}(f(S(\nu^{-1}_{T(a_1)}(\theta_1(a_1)))(\theta_2(T(a_1)))^{-1}, a_2 ))) \big)  (\sigma_{T(a_2)}(S(\nu^{-1}_{T(a_1)}(\theta_1(a_1)))(\theta_2(T(a_1)))^{-1}))^{-1}\notag\\   
		&&	\big(S(\nu^{-1}_{T(a_2)}(\theta_1(a_2)))(\theta_2(T(a_2)))^{-1}\big)^{-1} \big( \theta_2(T(a_2))(\theta_2(T(a_1) T(a_2)))^{-1}\sigma_{T(a_2)}(\theta_2(T(a_1))) \big)^{-1}.
	\end{eqnarray*}
	Using that $\nu^{-1}_{T(a_1)}(\mu_{\beta_{T(a_1)}(a_2)}(\theta_1(a_1)))= \mu_{a_2}(\nu^{-1}_{T(a_1)}(\theta_1(a_1)))$, the only part that survive is  
	\begin{align*}
		S\big(\nu^{-1}_{T(a_2)} \big(f(S(\nu^{-1}_{T(a_1)}(\theta_1(a_1))), a_2) \big)\big)\, S\big(\nu^{-1}_{T(a_2)} \big(\mu_{a_2}(\nu^{-1}_{T(a_1)}(\theta_1(a_1))) \big)\big) \,\big(\sigma_{T(a_2)} \big(S(\nu^{-1}_{T(a_1)}(\theta_1(a_1))) \big)\big)^{-1},
	\end{align*}
	which is 1 from \eqref{Rallcondn}  by putting $\nu^{-1}_{T(a_1)}(\theta_1(a_1))=k$. This completes the proof of the lemma.
\end{proof}

In view of Lemma \ref{rrb coboundary condition}, we define the first and the second cohomology group of $\mathcal{A}=(A, B, \beta, T)$  with coefficients in $\mathcal{K}=(K, L, \alpha, S)$ by $$\Ho^1_{RRB}(\mathcal{A}, \mathcal{K})=\Ker(\partial^1_{RRB})/  \IM(\partial^0_{RRB})$$
and $$\Ho^2_{RRB}(\mathcal{A}, \mathcal{K})=\Ker(\partial^2_{RRB})/  \IM(\partial^1_{RRB}).$$

\begin{question}
	Can the preceding construction be extended to develop a full cohomology theory for relative Rota-Baxter groups?
\end{question}

\medskip

\subsection{Equivalent extensions and second cohomology}\label{subsec equivalent extensions and cohomology}
Next, we define an equivalence of extensions of relative Rota-Baxter groups.

\begin{defn}
Two extensions of relative Rota-Baxter groups 
	$$\mathcal{E}_1 : \quad  {\bf 1} \longrightarrow (K,L,\alpha,S ) \stackrel{(i_1, i_2)}{\longrightarrow}  (H,G, \phi, R) \stackrel{(\pi_1, \pi_2)}{\longrightarrow} (A,B, \beta, T) \longrightarrow {\bf 1}$$
	and $$\mathcal{E}_2 : \quad {\bf 1} \longrightarrow (K,L,\alpha,S ) \stackrel{(i^\prime_1, i^\prime_2)}{\longrightarrow}  (H^\prime,G^\prime, \phi^\prime, R^\prime) \stackrel{(\pi^\prime_1, \pi^\prime_2)}{\longrightarrow} (A,B, \beta, T) \longrightarrow {\bf 1}$$
are said to be equivalent if there exists an isomorphism $$(\eta, \zeta):(H,G, \phi, R) \longrightarrow (H^\prime,G^\prime, \phi^\prime, R^\prime) $$
of relative Rota-Baxter groups  such that the following diagram commutes
	\begin{align}
	\begin{CD}\label{eqact}
		{\bf 1} @>>> (K,L,\alpha,S ) @>(i_1, i_2)>> (H,G, \phi, R) @>{{(\pi_1, \pi_2)} }>> (A,B, \beta, T) @>>> {\bf 1}\\ 
		&& @V{(\Id_K, ~\Id_L)} VV@V{(\eta, \zeta)} VV @V{(\Id_A, ~\Id_B)}VV \\
		{\bf 1} @>>> (K,L,\alpha,S ) @>(i^\prime_1, i^\prime_2)>>(H^\prime,G^\prime, \phi^\prime, R^\prime) @>{(\pi^\prime_1, \pi^\prime_2) }>> (A,B, \beta, T) @>>> {\bf 1}.
	\end{CD}
\end{align}
\end{defn}

\begin{prop}\label{eqactprop}
Equivalent extensions of relative Rota-Baxter groups induce identical associated actions.
\end{prop}

\begin{proof}
We already observed in Proposition \ref{construction of actions} and Lemma \ref{properties of f} that the actions defined by distinct  set-theoretic  sections to a given extension are the same.
 Suppose that $$\mathcal{E}_1 : \quad {\bf 1} \longrightarrow (K,L,\alpha,S ) \stackrel{(i_1, i_2)}{\longrightarrow}  (H,G, \phi, R) \stackrel{(\pi_1, \pi_2)}{\longrightarrow} (A,B, \beta, T) \longrightarrow {\bf 1}$$
	and $$\mathcal{E}_2 : \quad {\bf 1} \longrightarrow (K,L,\alpha,S ) \stackrel{(i^\prime_1, i^\prime_2)}{\longrightarrow}  (H^\prime,G^\prime, \phi^\prime, R^\prime) \stackrel{(\pi^\prime_1, \pi^\prime_2)}{\longrightarrow} (A,B, \beta, T) \longrightarrow {\bf 1}$$
are equivalent extensions. Let $(\nu, \mu, \sigma, f)$ and $(\nu^\prime, \mu^\prime, \sigma^\prime, f^\prime)$ be associated actions of the extensions $\mathcal{E}_1$ and $\mathcal{E}_2$, respectively. The commutativity of the diagram \eqref{eqact} implies that the group extensions $H$ and $H'$ of $A$ by $K$ are equivalent. It follows from the classical extension theory of groups that $\mu= \mu^\prime$ and $\sigma=\sigma^\prime$ (see \cite{MR0672956}). Let  $s_G: B \rightarrow G$ be a set-theoretic section to the group extension $G$. It can be seen that the composition $\zeta s_G$  defines a set-theoretic section to the group extension $G^\prime$. Furthermore, by  the commutativity of the diagram \eqref{eqact}, the action defined in equation \eqref{nuact} induced by $\zeta s_G$ coincides with the $\nu$.  By a similar reasoning, we prove that $f=f^\prime$, which completes the proof.
\end{proof}

Let $\mathcal{A}= (A,B, \beta, T)$ be an arbitrary relative Rota-Baxter group and $\mathcal{K}=  (K,L,\alpha,S )$ a trivial relative Rota-Baxter group, where $K$ and $L$ are abelian groups. Let $\Ext(\mathcal{A}, \mathcal{K})$ denote the set of all equivalence classes of extensions of $\mathcal{A}$ by $\mathcal{K}$. By Proposition \ref{eqactprop}, each element of  $\Ext(\mathcal{A}, \mathcal{K})$ gives rise to a unique quadruple $(\nu, \mu, \sigma,f)$ of actions. Thus, we can employ the term associated action of an element of $\Ext(\mathcal{A}, \mathcal{K})$. Further, let  $\Ext_{(\nu, \mu, \sigma, f)}(\mathcal{A}, \mathcal{K})$ denote the set of equivalence classes of extensions of $\mathcal{A}$ by $\mathcal{K}$ for which the associated action is $(\nu, \mu, \sigma, f)$. Then we can write
$$\Ext(\mathcal{A}, \mathcal{K})=\bigsqcup_{(\nu, \mu, \sigma, f)} \Ext_{(\nu, \mu, \sigma, f)}(\mathcal{A}, \mathcal{K}).$$

Suppose that  
$$\mathcal{E} : \quad {\bf 1} \longrightarrow (K,L, \alpha,S ) \stackrel{(i_1, i_2)}{\longrightarrow}  (H,G, \phi, R) \stackrel{(\pi_1, \pi_2)}{\longrightarrow} (A,B, \beta, T) \longrightarrow {\bf 1}$$
is an abelian extension of relative Rota-Baxter groups and $(s_H, s_G)$ is a  set-theoretic section to $\mathcal{E}$. Let  $\tau_1, \tau_2, \rho$ and $\chi$ be the maps defined in  \eqref{mucocycle}, \eqref{sigmacocycle}, \eqref{atilde} and \eqref{RRB2}, respectively, corresponding to set-theoretic sections $s_H$ and $s_G$. It follows from \eqref{iso3}, \eqref{rho second equation} and \eqref{R1} that $\partial_{RRB}^2(\tau_1, \tau_2, \rho, \chi)$ is trivial. Hence, the quadruple $(\tau_1, \tau_2, \rho, \chi)$ is a 2-cocycle induced by the set-theoretic section $(s_H, s_G)$.

\begin{thm}\label{ext and cohom bijection}
Let $\mathcal{A}= (A,B, \beta, T)$ be a relative Rota-Baxter group and $\mathcal{K}=  (K,L,\alpha,S )$ a trivial relative Rota-Baxter group, where $K$ and $L$ are abelian groups.  Let  $(\nu, \mu, \sigma, f)$ be the quadruple of actions that makes $\mathcal{K}$ into an $\mathcal{A}$-module. Then there is a bijection between $\Ext_{(\nu, \mu, \sigma, f)}(\mathcal{A}, \mathcal{K})$ and $\Ho^2_{RRB}(\mathcal{A}, \mathcal{K})$.  
\end{thm}

\begin{proof}
Let $\mathcal{A}= (A,B, \beta, T) $ and $\mathcal{K}=  (K,L,\alpha,S )$. Let $\mathcal{E}$ be an extension representing an element of  $\Ext_{(\nu, \mu, \sigma, f)}(\mathcal{A}, \mathcal{K})$, where
$$\mathcal{E} : \quad {\bf 1} \longrightarrow (K,L,\alpha,S ) \stackrel{(i_1, i_2)}{\longrightarrow}  (H,G, \phi, R) \stackrel{(\pi_1, \pi_2)}{\longrightarrow} (A,B, \beta, T) \longrightarrow {\bf 1}$$
We define
$$\Phi: \Ext_{(\nu, \mu, \sigma, f)}(\mathcal{A}, \mathcal{K}) \longrightarrow \Ho^2_{RRB}(\mathcal{A}, \mathcal{K}) \quad \textrm{by} \quad \Phi \big([\mathcal{E}]\big)= [(\tau_1, \tau_2, \rho, \chi )] ,$$
where$(\tau_1, \tau_2, \rho, \chi)$ is the 2-cocycle induced by some fixed  set-theoretic section $(s_H, s_G)$ to $\mathcal{E}$. We claim that $\Phi$ is well-defined. Thus, we need to show that $\Phi$ is independent of the choice of a section to $\mathcal{E}$, as well as the choice of a representative of the equivalence class $[\mathcal{E}]$.  Let $(s_H, s_G)$ and $(s_H^\prime,  s_G^\prime)$ be two distinct set-theoretic sections of  $\mathcal{E}$. Two sections of a group extension of  $A$ by $K$ differ by an element of  $K$. Similarly, two sections of a group extension of  $B$ by $L$ differ by an element of  $L$. Thus, there exist  maps $\theta_1: A \rightarrow K$ and $\theta_2: B \rightarrow L$ such that $s_H(a)s^\prime_H(a)^{-1}=\theta_1(a)$ and $s_G(b)s^\prime_G(b)^{-1}=\theta_2(b)$ for all $a \in A$ and $ b \in B$. Let $(\tau_1, \tau_2, \rho, \chi )$ and $(\tau^\prime_1, \tau^\prime_2, \rho^\prime, \chi^\prime )$ be 2-cocycles induced by $(s_H, s_G)$ and $(s_H^\prime,  s_G^\prime)$, respectively. Direct calculations yield the following:
\begin{enumerate}
\item  $\tau_1(a_1,a_2) \,(\tau^\prime_1(a_1, a_2))^{-1}=\partial^{1}_{\mu}(\theta_1)(a_1, a_2)$ for all $a_1, a_2 \in A$,
\item  $\tau_2(b_1, b_2) \,(\tau^\prime_2(b_1, b_2))^{-1}=\partial^{1}_{\sigma}(\theta_1)(b_1, b_2)$ for all $b_1, b_2 \in B$,
\item  $\rho(a, b) \,(\rho^\prime(a,b))^{-1}=\nu_b(f(\theta_2(b), a)\theta_1(a)) \,(\theta_1(\beta{_b(a)}))^{-1}$ for all $a \in A$ and $b \in B$,
\item $\chi(a) \,(\chi^\prime(a))^{-1} = S(\nu^{-1}_{T(a)}(\theta_1(a))) \,(\theta_2(T(a)))^{-1}$ for all $a \in A$.
\end{enumerate}
This gives $$(\tau_1, \tau_2, \rho, \chi ) \,(\tau^\prime_1, \tau^\prime_2, \rho^\prime, \chi^\prime )^{-1}=\partial^1_{RRB}(\theta_1, \theta_2).$$
Hence, the cohomology class of a 2-cocycle is independent of the choice of a defining section. Likewise, we can prove that equivalent extensions possess cohomologous 2-cocycles. This proves our claim that $\Phi$ is well-defined. To demonstrate the bijectivity of $\Phi$, we define its inverse explicitly. 

Consider a 2-cocycle $(\tau_1, \tau_2, \rho, \chi) \in \Ker(\partial_{RRB}^2)$. Define $H=A\times_{\tau_1} K$ and $G= B \times_{\tau_2} L$ to be the group extensions of $A$ by $K$ and $B$ by $L$ associated to the group 2-cocycles $\tau_1$ and $\tau_2$, respectively. Further, define $\phi: G \rightarrow \Aut(H)$ by 
\begin{align}
\phi_{(b,l)}(a,k)=\big(\beta{_{b}(a)}, ~\rho(a,b) \,\nu_b(f(l,a)k)\big)
\end{align} 
and $R: H \rightarrow G$ by 
\begin{align}
R(a,k)=\big(T(a), ~\chi(a)\,S(\nu^{-1}_{T(a)}(k))\big).
\end{align}
Using the fact that $(\tau_1, \tau_2, \rho, \chi) \in \Ker(\partial_{RRB}^2)$, it follows that $(H, G, \phi, R)$ is a relative Rota-Baxter group and is an extension of  $(A,B, \beta, T)$ by  $(K,L, \alpha,S )$ denoted by 
$$\mathcal{E}(\tau_1, \tau_2, \rho, \chi) : \quad {\bf 1} \to (K,L, \alpha,S ) \stackrel{(i_1, i_2)}{\longrightarrow}  (H, G, \phi, R) \stackrel{(\pi_1, \pi_2)}{\longrightarrow} (A,B, \beta, T) \to {\bf 1}.$$
Define 
$$\Psi:  \Ho^2_{RRB}(\mathcal{A}, \mathcal{K}) \longrightarrow \Ext_{(\nu, \mu, \sigma, f)}(\mathcal{A}, \mathcal{K}) \quad \textrm{by} \quad \Psi([(\tau_1, \tau_2, \rho, \chi)])=[\mathcal{E}(\tau_1, \tau_2, \rho, \chi)].$$

We claim  that $\Psi$ is the inverse of $\Phi$. To substantiate this claim, it is necessary to establish the well-definedness of $\Psi$. Let $(\tau_1, \tau_2, \rho, \chi )$ and $(\tau^\prime_1, \tau^\prime_2, \rho^\prime, \chi^\prime ) \in \Ker(\partial_{RRB}^2)$ be two cohomologous 2-cocycles. By definition, there exist  maps $\theta_1: A \rightarrow K$ and $\theta_2: B \rightarrow L$ such that
$$(\tau_1, \tau_2, \rho, \chi ) \, (\tau^\prime_1, \tau^\prime_2, \rho^\prime, \chi^\prime )^{-1}=\partial^1_{RRB}(\theta_1, \theta_2).$$
Let $(A\times_{\tau_1} K,B \times_{\tau_2} L, \phi, R)$ and $(A\times_{\tau^\prime_1} K,B \times_{\tau^\prime_2} L, \phi^\prime, R^\prime)$ be the relative Rota-Baxter groups induced by the 2-cocycles $(\tau_1, \tau_2, \rho, \chi )$ and $(\tau^\prime_1, \tau^\prime_2, \rho^\prime, \chi^\prime )$, respectively. Define $\gamma_1: A \times_{\tau_1} K \rightarrow A \times_{\tau^{\prime}_1} K$ by $\gamma_1(a,k) = (a, k\theta_1(a))$ and $\gamma_2: B \times_{\tau_2} L \rightarrow B \times_{\tau^{\prime}_2} L$ by $\gamma_2(b,l) = (b, l\theta_2(a))$. Straightforward calculations show that $(\gamma_1, \gamma_2)$ is an isomorphism between relative Rota-Baxter groups $(A\times_{\tau_1} K,B \times_{\tau_2} L, \phi, R)$ and $(A\times_{\tau^\prime_1} K,B \times_{\tau^\prime_2} L, \phi^\prime, R^\prime)$ such that the following diagram commutes
 \begin{align*}
 	\begin{CD}
 		{\bf 1} @>>> (K,L,\alpha,S ) @>(i_1, i_2)>> (A\times_{\tau_1} K,B \times_{\tau_2} L, \phi, R) @>{{(\pi_1, \pi_2)} }>> (A,B, \beta, T) @>>> {\bf 1}\\ 
 		&& @V{(\Id_K, ~\Id_L)} VV@V{(\gamma_1, \gamma_2)} VV @V{(\Id_A, ~\Id_B) }VV \\
 		{\bf 1} @>>> (K,L,\alpha,S ) @>(i^\prime_1, i^\prime_2)>>(A\times_{\tau^\prime_1} K,B \times_{\tau^\prime_2} L, \phi^\prime, R^\prime)@>{(\pi^\prime_1, \pi^\prime_2) }>> (A,B, \beta, T) @>>> {\bf 1}.
 	\end{CD}
 \end{align*}
This shows that  $\mathcal{E}(\tau_1, \tau_2, \rho, \chi)$ and $\mathcal{E}(\tau^\prime_1, \tau^\prime_2, \rho^\prime, \chi^\prime)$ are equivalent extensions, and hence the map $\Psi$ is well-defined. By direct inspection, one can readily verify that  $\Phi$ and $\Psi$ are inverse of each other. This concludes the proof.
\end{proof}

The following is immediate from the definition.

\begin{prop}
 Let $\mathcal{A}=(A, B, \beta, T)$ be a  relative Rota-Baxter group and $\mathcal{K}=(K, L, \alpha, S)$ a module over  $\mathcal{A}$ via the action $(\nu, \mu, \sigma, f)$. Then  $\Ho^0_{RRB}(\mathcal{A}, \mathcal{K})$ is the fixed-point set of actions $\nu$, $\mu$, and $\sigma$, that is,
$$\Ho^0_{RRB}(\mathcal{A}, \mathcal{K})=\big\{(k,l) \in K \times L ~\mid~ \mu_a(k)=k,~ \nu_b(k)=k, ~\sigma_b(l)=l  ~\mbox{ for all } a \in A \mbox{ and }b \in B \big\}.$$
\end{prop}

\begin{remark}
If $\mathcal{A}= (1,B, 1, 1)$  and $\mathcal{K}=  (1,L,1, 1)$ then $\Ho^2_{RRB}(\mathcal{A}, \mathcal{K}) \cong \Ho^2(B,L)$. Similarly, if  $\mathcal{A}= (A,1, 1, 1)$ and $\mathcal{K}=  (K,1,1, 1)$ then $\Ho^2_{RRB}(\mathcal{A}, \mathcal{K}) \cong \Ho^2(A,K)$, where $\Ho^2$ denotes the second cohomology of groups. 
\end{remark}

\section{Extensions and cohomology of skew left braces}\label{sec extensions skew braces}

We recall some necessary results on extensions of skew left braces by abelian groups \cite{NMY1}.  Let  $(M, \cdot, \circ_M)$ be a skew left brace and 
$$\mathcal{E}: \quad {\bf 1} \longrightarrow I \stackrel{i_1}{\longrightarrow}  E \stackrel{\pi_1}{\longrightarrow} M \longrightarrow {\bf 1}$$ an extension of skew left braces such that $I$ is an abelian group viewed as a trivial brace.  The group $I$ is regarded as a subgroup of $E$, and the group operation in the additive group of $E$ will be denoted by juxtaposition.  Let $s : M \rightarrow E$ be a set-theoretic section to $\mathcal{E}$. Then, for all $ m \in M$ and $y \in I$, we  define  maps $\xi, 	\epsilon : M^{(\circ)} \rightarrow \Aut(I)$ and $\zeta :M^{(\cdot)} \rightarrow \Aut(I)$  by
\begin{eqnarray}
	\xi_m(y) & = &  \lambda^E_{s(m)}(y),\label{action1 sb}\\
	\zeta_m(y) & = & s(m)^{-1}  \,y \, s(m),\label{action2 sb}\\
	\epsilon_m(y) & = & s(m)^{\dagger} \circ_E y \circ_E s(m),\label{action3 sb}
\end{eqnarray}
for $m \in M$ and $y \in I$, where $x^{-1}$ and $x^\dagger$ denotes the inverse of $x$ in $E^{(\circ)}$ and $E^{(\cdot)}$, respectively. It is not difficult to see  that the map $\xi$ is  a homomorphism, whereas the maps $\zeta, \epsilon $ are  anti-homomorphisms. Furthermore, these maps are independent of the choice of the set-theoretic section \cite[Proposition 3.4]{NMY1}. The triplet $(\xi, \zeta, \epsilon)$ is called the \emph{associated action} of the extension $\mathcal{E}$. 
\par 

Next, we recall the definition of the second cohomology group of a skew left brace $(M, \cdot, \circ)$ with coefficients in an abelian group $I$ viewed as a trivial brace.  Let  $\xi: M ^{(\circ)} \rightarrow \Aut (I)$ be a homomorphism and   $ \zeta: M^{(\cdot)} \rightarrow \Aut(I)$  and $\epsilon : M^{(\circ)} \rightarrow \Aut(I)$ be anti-homomorphisms satisfying the following conditions
\begin{eqnarray*}
	\xi_{m_1 \cdot m_2}(\epsilon_{m_1 \cdot m_2}(y)) \,\zeta_{m_2}(y) & = & \zeta_{m_1}(\xi_{ m_1 }(\epsilon_{m_1}(y))) \, \xi_{ m_2} (\epsilon_{m_2}(y)),\\
	\zeta_{m^{-1}_1 \cdot (m_1 \circ_M m_2)}(\xi_{m_1}(y))  & =& \xi_{m_1}(\zeta_{m_2}(y)),
\end{eqnarray*}
for all $m_1, m_2 \in M$ and $y \in I$. In \cite{NMY1}, such a triplet is referred as a  \emph{good triplet} of action of $H$ on $I$.
\par 

Let $g,f: M \times M \rightarrow I$ be maps satisfying
\begin{small}
	\begin{align} 
		g(m_2, m_3)  \,g(m_1 \cdot m_2, m_3)^{-1} \,g(m_1, m_2 \cdot m_3) \,\zeta_{m_3}( g(m_1, m_2))^{-1} =&\;  1 \label{sbcocycle1},\\
		\xi_{m_1}(f(m_2, m_3)) \, f(m_1 \circ_M m_2, m_3)^{-1} \, f(m_1, m_2 \circ_M m_3) \, \xi_{m_1 \circ_M m_2 \circ_M m_3} (\epsilon_{m_3}(\nu^{-1}_{m_1 \circ_M m_2}f(m_1, m_2)))^{-1} 	= & \;  1, \label{sbcocycle2}\\
		& \notag\\
		\xi_{m_1}(g(m_2, m_3))\, \zeta_{m_1 \circ_M m_3}(g(m_1, m^{-1}_1)) \,\zeta_{m_1 \circ_M m_3}(g(m_1 \circ_M m_2,  m^{-1}_1))^{-1} \label{sbcocycle3} \\ 
		g((m_1 \circ_M m_2) m^{-1}_1, m_1 \circ_M m_3)^{-1}\, \zeta_{- m_1\cdot (m_1 \circ_M m_3)}\, (f(m_1, m_2))^{-1} \, f(m_1, m_2 \cdot m_3)   f(m_1, m_3)^{-1} = &\; 1, \notag
	\end{align}
\end{small}
for all $m_1, m_2, m_3 \in M$.  Let

\begin{equation*}
	\Z_N^2(M, I) = \Big\{ (g ,f) \quad \Big \vert  \quad g,f:M \times M \rightarrow I ~ \textrm{satisfy}~\eqref{sbcocycle1}, \eqref{sbcocycle2}, \eqref{sbcocycle3} ~\textrm{and}~ \textrm{vanish on degenerate tuples}  \Big\},
\end{equation*}
and $\B_N^2(M, I)$ be the collection of the pairs $(g, f) \in  \Z_N^2(M, I)$ such that there exists a map $\theta:M \to I$ satisfying
\begin{eqnarray*}
	g(m_1, m_2) &=& \xi_{m_1\cdot m_2}(\theta(m_1\cdot m_2)^{-1}) ~\zeta_{m_2}((\xi_{m_1}(\theta(m_1))))~ \xi_{m_2}(\theta(m_2)),\\
	f(m_1, m_2) &=& \theta(m_1 \circ m_2)^{-1}  ~\epsilon_{m_2}(\theta(m_1)) ~ \theta(m_2),
\end{eqnarray*}
for all $m_1, m_2 \in M$.  Then the second cohomology group of  $(M, \cdot, \circ_M)$ with coefficients in $I$ corresponding to the given good triplet of actions $(\nu, \mu, \sigma)$ is defined as $\Ho^2_N(M, I) =  \Z_N^2(M, I)/\B_N^2(M, I)$.
\par

Let $\Ext_{(\xi, \zeta, \epsilon)}(M, I)$ denote the set of equivalence classes of those skew left brace extensions of $M$ by $I$ whose corresponding triplet of actions is $(\xi, \zeta, \epsilon)$. It is proven in \cite[Proposition 3.4]{NMY1} that $(\xi, \zeta, \epsilon)$ forms a good triplet of actions for $M$ on $I$ and that the following holds \cite[Theorem A]{NMY1}.

\begin{thm}\label{gbij-thm sb}
	Let $(M, \cdot, \circ)$ be a skew left brace and $(I, +)$ an abelian group viewed as a trivial brace. Then there is a bijection $\Lambda:\Ext_{(\xi, \zeta, \epsilon)}(M, I) \rightarrow \Ho^2_N(M, I)$ given by $\Lambda([\mathcal{E}])=[(\tau, \tilde{\tau})]$, where 
	\begin{eqnarray*}
		\tau(m_1, m_2) &= &  s(m_1 \cdot m_2)^{-1}  s(m_1)  s(m_2),\\
		\tilde{\tau}(m_1, m_2) &=& s(m_1 \circ m_2)^{-1}  \,(s(m_1) \circ  s(m_2)),
	\end{eqnarray*}
	and $s$ is a set-theoretic section to $\mathcal{E}$.
\end{thm}

Suppose that 
\begin{equation*}
\mathcal{E} : \quad {\bf 1} \longrightarrow (K,L,\alpha,S ) \stackrel{(i_1, i_2)}{\longrightarrow}  (H,G, \phi, R) \stackrel{(\pi_1, \pi_2)}{\longrightarrow} (A,B, \beta, T) \longrightarrow {\bf 1}
\end{equation*}
 is an abelian extension of relative Rota-Baxter groups and 
\begin{equation*}
\mathcal{E}_{SB}: \quad  {\bf 1} \longrightarrow K_{S} \stackrel{i_1}{\longrightarrow}  H_{R} \stackrel{\pi_1}{\longrightarrow} A_{T} \longrightarrow {\bf 1}
\end{equation*}
 the induced skew left brace extension. Let $(s_H, s_G)$ be a set-theoretic section to $\mathcal{E}$, and  $(\xi, \zeta, \epsilon)$ the associated action of the extension $\mathcal{E}_{SB}$. For each $a \in A$ and $k \in K$, it is easy to see that  $$\zeta_a= \mu_a$$ and 
$$\xi_a(k)= \phi_{R(s_H(a))}(k)=\phi_{s_G(T(a))\chi(a)}(k)=\nu_{T(a)}(k)$$ 
since $\phi_{\chi(a)}$ acts as the identity.  To determine the map $\epsilon$, we have
	\begin{eqnarray*}
		\epsilon_a(k) &=& s_H(a)^{\dagger} \circ_R k \circ_R s_H(a)\\
		&=& (s_H(a)^{\dagger} \circ_R k)~ \phi_{R(s_H(a))^{-1} R(k)}(s_H(a)), \quad \textrm{by definition of $\circ_R$}\\
		&=& \phi_{R(s_H(a))^{-1}} (s_H(a)^{-1})~ \phi_{R(s_H(a))^{-1}}(k) ~\phi_{R(s_H(a))^{-1}} \phi_{ R(k)}(s_H(a)),
	\end{eqnarray*}
by definition of $\circ_R$ and the facts that $s_H(a)^{\dagger}=\phi_{R(s_H(a))^{-1}} (s_H(a)^{-1})$ and $\phi_{R(s_H(a)^{\dagger})}=\phi_{R(s_H(a))^{-1}}$. Now using the facts that $R(s_H(a))^{-1}= \chi(a)^{-1} s_G(T(a))^{-1}$, $\phi_{ R(k)}(s_H(a))=s_H(a)f(R(k), a)$ and \eqref{sigmaact},  we have 
	\begin{eqnarray*}
		\epsilon_a(k) &=& \phi_{\chi(a)^{-1} s_G(T(a))^{-1}} (s_H(a))^{-1} \;  \nu_{T(a)^{-1}}(k) \; \phi_{\chi(a)^{-1} s_G(T(a))^{-1}}(s_H(a) f(R(k), a))\notag \\
		&=&  \phi_{s_G(T(a))^{-1} \sigma_{T(a)^{-1}}(\chi(a)^{-1})} (s_H(a))^{-1}\; \nu_{T(a)^{-1}}(k) \; \phi_{ s_G(T(a))^{-1} \sigma_{T(a)^{-1}}(\chi(a)^{-1})}(s_H(a)) \; \nu_{T(a)^{-1}}(f(R(k), a)).
	\end{eqnarray*}
	Seting $k_1:= \nu_{T(a)^{-1}}(k)$, $k_2:= \nu_{T(a)^{-1}}(f(R(k), a))$ and $l:= \sigma_{T(a)^{-1}}(\chi(a)^{-1})$,  we have
	\begin{eqnarray*}
		\epsilon_a(k) &=&  \phi_{s_G(T(a))^{-1} l} (s_H(a))^{-1} \; k_1 \; \phi_{ s_G(T(a))^{-1} l}(s_H(a)) \; k_2  \\
		&=&  \big(\phi_{s_G(T(a))^{-1}} \phi_l (s_H(a)) \big)^{-1} k_1 \; \phi_{s_G(T(a))^{-1}} \phi_l (s_H(a)) \; k_2 \\
		&=&  \big(\phi_{s_G(T(a))^{-1}} (s_H(a)f(l, a))\big)^{-1} k_1 \; \phi_{s_G(T(a))^{-1}} (s_H(a)f(l, a)) \; k_2\\
		&=& \big(\phi_{s_G(T(a))^{-1}} (s_H(a))\; \nu_{T(a)^{-1}}(f(l,a))\big)^{-1} k_1 \; \phi_{s_G(T(a))^{-1}} (s_H(a))\; \nu_{T(a)^{-1}}(f(l,a)) k_2.
	\end{eqnarray*}
	Using that $s_G(T(a))^{-1}=s_G(T(a)^{-1}) \tau_2(T(a), T(a)^{-1})^{-1}$, we have
	\begin{eqnarray*}
		\epsilon_a(k) &=& \big(\phi_{s_G(T(a)^{-1})  \tau_2(T(a), T(a)^{-1})^{-1} } (s_H(a))\; \nu_{T(a)^{-1}}(f(l,a))\big)^{-1} k_1 \; \phi_{s_G(T(a)^{-1})  \tau_2(T(a), T(a)^{-1})^{-1} } (s_H(a))\\
		&& \nu_{T(a)^{-1}}(f(l,a)) k_2\\
		&=& \big(s_H (\beta_{T(a)^{-1}}(a))\; \rho(a, T(a)^{-1}) \; f(\tau_2(T(a), T(a)^{-1})^{-1}, a)\; \nu_{T(a)^{-1}}(f(l,a))\big)^{-1}\; k_1 \; s_H (\beta_{T(a)^{-1}}(a))\\
		&&  \rho(a, T(a)^{-1}) \; f(\tau_2(T(a), T(a)^{-1})^{-1}, a)\; \nu_{T(a)^{-1}}(f(l,a))\; k_2\\
		&=& \nu_{T(a)^{-1}}(f(l,a))^{-1}\; f(\tau_2(T(a), T(a)^{-1})^{-1}, a)^{-1}\;\rho(a, T(a)^{-1})^{-1}\; s_H (\beta_{T(a)^{-1}}(a))^{-1}\;  k_1 \; s_H (\beta_{T(a)^{-1}}(a))\\
		&&  \rho(a, T(a)^{-1})\;  f(\tau_2(T(a), T(a)^{-1})^{-1}, a)\; \nu_{T(a)^{-1}}(f(l,a))\; k_2\\
		&=& \mu_{\beta_{T(a)^{-1}}(a)}(k_1) \; k_2.
	\end{eqnarray*}
	Using the values of $k_1$ and $k_2$, we have
	\begin{eqnarray}\label{epsilonf}
		\epsilon_a(k) &=&  \mu_{\beta_{T(a)^{-1}}(a)}(\nu_{T(a)^{-1}}(k)) \; \nu_{T(a)^{-1}}(f(R(k), a)).
	\end{eqnarray}
The preceding computations show that $\epsilon$ is completely determined by $\nu, \mu$ and  $f$. Hence, we conclude that the associated action of the induced skew left brace extension is completely determined by the action of the given relative Rota-Baxter group. This together with Proposition \ref{eqactprop} gives the following corollary.

\begin{cor}\label{RRB2sbext}
	Let $\mathcal{A}= (A,B, \beta, T)$ be a relative Rota-Baxter group and $\mathcal{K}=  (K,L,\alpha,S )$ a trivial relative Rota-Baxter group such that $K$ and $L$ are abelian groups. Let  
	$$\mathcal{E} : \quad {\bf 1} \longrightarrow (K,L,\alpha,S ) \stackrel{(i_1, i_2)}{\longrightarrow}  (H,G, \phi, R) \stackrel{(\pi_1, \pi_2)}{\longrightarrow} (A,B, \beta, T) \longrightarrow {\bf 1}$$
	be an extension of relative Rota-Baxter groups with action $(\nu, \mu, \sigma, f)$ and 
	$$\mathcal{E}_{SB}: \quad {\bf 1} \longrightarrow K_{S} \stackrel{i_1}{\longrightarrow}  H_{R} \stackrel{\pi_1}{\longrightarrow} A_{T} \longrightarrow {\bf 1}$$
	the induced skew left brace extension with associated action $(\xi, \zeta, \epsilon)$.  Then there is a map $\Pi: \Ext_{(\nu, \mu, \sigma, f)}(\mathcal{A}, \mathcal{K}) \to \Ext_{(\xi, \zeta, \epsilon)}(A_T, K_S)$ given by  $\Pi\big([\mathcal{E}] \big)= [\mathcal{E}_{SB}]$.
\end{cor}

For bijective relative Rota-Baxter groups, we can work in the reverse direction.

\begin{prop}\label{actnprop}
Let  $\mathcal{A}=(A, B, \beta, T)$ be a  relative Rota-Baxter group and $\mathcal{K}=(K, L, \alpha, S)$ a trivial relative Rota-Baxter group such that $T$ and $S$ are bijections and $K$ and $L$ are abelian.
 Let
\begin{equation}\label{extension of SLB for action}
\mathcal{E}_{SB}: \quad {\bf 1} \longrightarrow K_{S} \stackrel{i_1}{\longrightarrow}  E \stackrel{\pi_1}{\longrightarrow} A_{T} \longrightarrow {\bf 1}
\end{equation}
be an extension of skew left braces with associated action $(\xi, \zeta, \epsilon)$. If $\mathcal{E}_{SB}$ is induced by the extension 
\begin{equation}\label{extension of RRB for action}
\mathcal{E} : \quad {\bf 1} \longrightarrow (K,L,\alpha,S ) \stackrel{(i_1, i_2)}{\longrightarrow}  (E,G, \phi, R) \stackrel{(\pi_1, \pi_2)}{\longrightarrow} (A,B, \beta, T) \longrightarrow {\bf 1}
\end{equation}
of relative  Rota-Baxter groups, then the action $(\nu, \mu, \sigma, f)$ associated to $\mathcal{E}$ is given by 
\begin{eqnarray}
\nu_a  &=& \xi_{T^{-1}(a)}, \label{RRB2sbactn2}\\
\mu_a &=& \zeta_a, \label{RRB2sbactn1}\\
\sigma_b &=& S^{-1} \; \epsilon_{T^{-1}(b)} \; S, \label{RRB2sbactn3}\\
f(l,a) &=&  \zeta_a(S^{-1}(l^{-1})) \; \xi_a(\epsilon(S^{-1}(l))), \label{RRB2sbactn4}
\end{eqnarray}
for $a \in A$, $b \in B$ and $l \in L$.
\end{prop}

\begin{proof}
Let $(s_H, s_G)$ be a set-theoretic section to $\mathcal{E}$. Then  $\epsilon_a(k) = s_H(a)^{\dagger} \circ_R k \circ_R s_H(a)$ for $a \in A$ and $k \in K$. Applying $R$ on both the sides, and using the facts that $R|_K=S$ and $R: H^{(\circ_R)} \to G$ is a homomorphism, we obtain
	\begin{eqnarray*}
		S(\epsilon_a(k))&=&  R((s_H(a))^{-1} \,S(k)\, R((s_H(a))\\
		&=&s_G(T(a))^{-1} \,S(k)\,s_G(T(a)),  \quad \textrm{by}~ \eqref{RRB2}~\textrm{and the fact that $L$ is abelian}~\\
		&=& \sigma_{T(a)} (S(k)), \quad \textrm{by}~ \eqref{sigmaact}.
	\end{eqnarray*}
Thus, $S\,  \epsilon_a= \sigma_{T(a)} \, S$ for all $ a \in A$. If $(\xi, \zeta, \epsilon)$ is the  associated action of the skew left brace extension \eqref{extension of SLB for action} of $A_T$ by $K_S$, then the action $(\nu, \mu, \sigma, f)$ of the corresponding relative Rota-Baxter extension \eqref{extension of RRB for action} is given by 
$$ \mu_a = \zeta_a, \quad \nu_a  = \xi_{T^{-1}(a)} \quad \textrm{and} \quad  \sigma_b = S^{-1} \; \epsilon_{T^{-1}(b)} \; S,$$
for all $a \in A$ and $b \in B$. Calculating $f$ separately, we see from \eqref{epsilonf} that
\begin{eqnarray*}
	f(l,a) &=&  \nu_{T(a)}\big( \mu_{\beta_{T(a)^{-1}}(a)}(\nu_{T(a)^{-1}}(S^{-1}(l)))^{-1} \; \epsilon_a(S^{-1}(l))\big) \\
	&=& \nu_{T(a)}\big( \mu_{\beta_{T(a)^{-1}}(a)}(\nu_{T(a)^{-1}}(S^{-1}(l)))^{-1}\big) ~\nu_{T(a)}(\epsilon_a(S^{-1}(l)))\notag\\
	&=& \mu_{a}(S^{-1}(l^{-1})) \;	\nu_{T(a)}(\epsilon_a(S^{-1}(l))), \quad \textrm{using}~ \eqref{mcon}    \notag\\
	&=& \zeta_a(S^{-1}(l^{-1})) \; \xi_a(\epsilon(S^{-1}(l))),\notag
\end{eqnarray*}
for $l \in L$ and $a \in A$. 
\end{proof}
\medskip

Next, we find a relationship between $\Ho^2_{RRB}(\mathcal{A}, \mathcal{K})$ and $\Ho^2_N(A_R, K_S)$.  Define $\tilde{\tau}: A \times A \to K$ by $\tilde{\tau}(a_1, a_2) = s_H(a_1 \circ _T a_2)^{-1}(s_H(a_1) \circ_ R s_H(a_2))$. Then we have
\begin{eqnarray*}
	\tilde{\tau}(a_1, a_2) &=& s_H(a_1 \circ _T a_2)^{-1}(s_H(a_1) \circ_ R s_H(a_2))\\
	&=& s_H(a_1 \beta_{T(a_1)}(a_2))^{-1} (s_H(a_1) \phi_{R(s_H(a_1))}(s_H(a_2)))\\
	&=& s_H(a_1 \beta_{T(a_1)}(a_2))^{-1} (   s_H(a_1) \phi_{s_G(T(a_1)) \chi(a_1)}  (s_H(a_2))       )\\
	&=& s_H(a_1 \beta_{T(a_1)}(a_2))^{-1} (s_H(a_1)  \phi_{s_G(T(a_1))}(s_H(a_2) f(\chi(a_1), a_2 )  )  ), \quad \textrm{by}~ \eqref{fdefn} \\
	&=& s_H(a_1 \beta_{T(a_1)}(a_2))^{-1} (s_H(a_1) s_H(\beta_{T(a_1)}(a_2)) \rho(T(a_1), a_2) \nu_{T(a_1)}( f(\chi(a_1), a_2 ))), \quad \textrm{by}~  \eqref{atilde}\\
	&=& \tau_1(a_1, \beta_{T(a_1)}(a_2)) \rho(T(a_1), a_2) \nu_{T(a_1)}( f(\chi(a_1), a_2) )), \quad \textrm{by}~  \eqref{mucocycle}.
\end{eqnarray*}

Let $\Psi:  \Ho^2_{RRB}(\mathcal{A}, \mathcal{K}) \longrightarrow \Ext_{(\nu, \mu, \sigma, f)}(\mathcal{A}, \mathcal{K})$  and $\Lambda:\Ext_{(\xi, \zeta, \epsilon)}(M, I) \longrightarrow \Ho^2_N(M, I)$ be the bijections defined in  theorems \ref{ext and cohom bijection} and \ref{gbij-thm sb}, respectively. If $\Pi: \Ext_{(\nu, \mu, \sigma, f)}(\mathcal{A}, \mathcal{K}) \longrightarrow \Ext_{(\xi, \zeta, \epsilon)}(A_T, K_S)$ is the map defined in Corollary \ref{RRB2sbext}, then we have the map $\Lambda \Pi \Psi: \Ho^2_{RRB}(\mathcal{A}, \mathcal{K}) \to \Ho^2_N(A_T, K_S)$. In fact, the map $\Lambda \Pi \Psi$ is explicitly given by $\Lambda \Pi \Psi\big( [(\tau_1, \tau_2, \rho, \chi)]\big)= [\tau_1, \tau^{(\beta, T)}_1\rho^{T}\chi^{(T, f)} ]$, where 
\begin{eqnarray*}
	\tau^{(\beta, T)}_1(a_1, a_2) &=&  \tau_1(a_1, \beta_{T(a_1)}(a_2)),\\
	\rho^{T}(a_1, a_2) &=& \rho(T(a_1), a_2),\\
	\chi^{(T, f)}(a_1, a_2) & =& \nu_{T(a_1)}( f(\chi(a_1), a_2)),
\end{eqnarray*}
for all $a_1, a_2 \in A$. In fact, we have the following result. 

\begin{prop}\label{cohom RRB to cohom skew brace}
	Let $\mathcal{A}= (A,B, \beta, T)$ be a relative Rota-Baxter group and $\mathcal{K}=  (K,L,\alpha,S )$ a module over $\mathcal{A}$ with respect to the action $(\nu, \mu,\sigma, f)$. Then the map $\Lambda \Pi \Psi: \Ho^2_{RRB}(\mathcal{A}, \mathcal{K}) \to \Ho^2_N(A_R, K_S)$ is a homomorphism of groups.
\end{prop}
\begin{proof}
	Let $[(\tau_1, \tau_2, \rho, \chi)]$ and $[(\tau^\prime_1, \tau^\prime_2, \rho^\prime, \chi^\prime)]$   be elements in  $\Ho^2_{RRB}(\mathcal{A}, \mathcal{K})$. Then 
	\begin{eqnarray*}
		\Lambda \Pi \Psi( [(\tau_1, \tau_2, \rho, \chi)] \, [(\tau^\prime_1, \tau^\prime_2, \rho^\prime, \chi^\prime)] ) &=& \Lambda \Pi \Psi( [(\tau_1\tau^\prime_1, \tau_2\tau^\prime_2, \rho \rho^\prime, \chi\chi^\prime)])\\
		&=& [(\tau_1 \tau^\prime_1,  (\tau_1 \tau^\prime_1)^{(\beta, T)}  (\rho_1\rho^\prime_1)^T (\chi\chi^\prime)^{(T,f)})].
	\end{eqnarray*}
	It is easy to see that 
	\begin{eqnarray*}
		(\tau_1\tau^\prime_1)^{(\beta, T)}&=&\tau^{(\beta, T)}_1 (\tau^\prime_1)^{(\beta, T)},\\
		(\rho_1\rho^\prime_1)^T &=& \rho^T_1  (\rho^\prime_1)^T,\\
		(\chi\chi^\prime)^{(T,f)} & = & \chi^{(T, f)}  (\chi^\prime)^{(T, f)}, \quad \textrm{by linearity of $f$ in the first coordinate.}
	\end{eqnarray*}
	Hence, we have 
	\begin{eqnarray*}
		\Lambda \Pi \Psi( [(\tau_1, \tau_2, \rho, \chi)] [(\tau^\prime_1, \tau^\prime_2, \rho^\prime, \chi^\prime)] ) &=&[(\tau_1,\tau^{(\beta, T)}_1,  \rho^T_1, \chi^{(T, f)})] \,[(\tau^\prime_1, (\tau^\prime_1)^{(\beta, T)}, (\rho^\prime_1)^T,  (\chi^\prime)^{(T, f)}) ]\\
		& =& \Lambda \Pi \Psi( [(\tau_1, \tau_2, \rho, \chi)]) \,	\Lambda \Pi \Psi([(\tau^\prime_1, \tau^\prime_2, \rho^\prime, \chi^\prime)] ),
	\end{eqnarray*}
	which shows that $\Lambda \Pi \Psi$ is a homomorphism.
\end{proof}

Let $\mathcal{A} = (A, B, \beta, T)$ be a bijective relative Rota-Baxter group.  Then we have an isomorphism
$$(\Id_A, T): (A^{(\cdot)}, A^{(\circ_T)}, \beta \, T, \Id_A) \stackrel{\cong}{\longrightarrow}   (A,B, \beta, T).$$ 
Thus, without loss of generality, we consider a bijective relative Rota-Baxter group to be of the form $(A^{(\cdot)}, A^{(\circ_T)}, \beta \, T, \Id_A)$.

\begin{prop}
Let  $\mathcal{A}=(A, B, \beta, T)$ be a  relative Rota-Baxter group and $\mathcal{K}=(K, L, \alpha, S)$ a trivial relative Rota-Baxter group such that  $T$ and $S$ are bijections and $K$ and $L$ are abelian.  Then the map  $\Pi$ defined in Corollary \ref{RRB2sbext} is a bijection.
\end{prop}

\begin{proof} 
Since $\mathcal{A}$ and $\mathcal{K}$ are bijective  relative Rota-Baxter groups, in view of the preceding remark, we can take $\mathcal{A} =(A^{(\cdot)}, A^{(\circ_T)}, \beta \, T, \Id_A)$ and $\mathcal{K} = (K, K, \alpha, \Id_K)$. Further, the relationship between the actions  $(\nu, \mu, \sigma, f)$ and $(\xi, \zeta, \epsilon)$ is given by  \eqref{RRB2sbactn1},\eqref{RRB2sbactn2}, \eqref{RRB2sbactn3} and \eqref{RRB2sbactn4}. We first show that the map $\Pi$ is injective. Let $$\mathcal{E} : \quad {\bf 1} \longrightarrow (K,K,\alpha,\Id_K) \stackrel{(i_1, i_2)}{\longrightarrow}  (H,G, \phi, R) \stackrel{(\pi_1, \pi_2)}{\longrightarrow} (A^{(\cdot)},A^{(\circ_T)}, \beta\;T, \Id_A) \longrightarrow {\bf 1},$$ 
		$$\mathcal{E}^\prime : \quad {\bf 1} \longrightarrow (K,K,\alpha,\Id_K) \stackrel{(i^\prime_1, i^\prime_2)}{\longrightarrow}  (H^\prime,G^\prime, \phi^\prime, R^\prime
		) \stackrel{(\pi^\prime_1, \pi^\prime_2)}{\longrightarrow}(A^{(\cdot)},A^{(\circ_T)}, \beta\;T, \Id_A) \longrightarrow {\bf 1}$$ be two  extensions of $\mathcal{K}$ by $\mathcal{A}$ such that  the skew left brace extensions  $\mathcal{E}_{SB}$ and  $\mathcal{E}^\prime_{SB}$ are equivalent. Then we have  an isomorphism $\eta: H_R \longrightarrow H^\prime_{R^\prime}$ of skew left braces such that the following diagram commutes
		$$\begin{CD}
			\mathcal{E}_{SB}:\quad  {\bf 1} @>>> K_S @>{i_1}>>H_R  @>{{\pi_1} }>> A_T  @>>> {\bf 1} \\
			&&  @V{\Id_{K}}VV @V{\eta}VV @VV{{\Id_{A}}}V \\
			\mathcal{E}^\prime_{SB}: \quad {\bf 1} @>>> K_S @>{i^\prime_1}>> H^\prime_{R^\prime} @>{{\pi_1^\prime} }>> A_T @>>> {\bf 1} .
		\end{CD}$$
Since $\eta$ is an isomorphism of skew left braces,  we have $\eta \; \phi_{R(h)}=\phi^\prime_{R^\prime (\eta(h))} \; \eta$ for all $h \in H$. In view of Remark \ref{bijext}, both  $R$ and $R^\prime$ are bijections. Thus, we have an isomorphism
		$$(\eta, R^\prime \eta R^{-1} ) :  (H,G, \phi, R) \longrightarrow  (H^\prime,G^\prime, \phi^\prime, R^\prime )$$
of relative Rota-Baxter groups.  Further, the following diagram commutes 
		\begin{align}
			\begin{CD}\label{eqact1}
				{\bf 1} @>>> (K,K,\alpha,\Id_K) @>(i_1, i_2)>> (H,G, \phi, R) @>{{(\pi_1, \pi_2)} }>> (A^{(\cdot)},A^{(\circ_T)}, \beta \; T, \Id_A) @>>> {\bf 1}\\ 
				&& @V{(\Id_K, ~\Id_K)} VV@V{(\eta, R^\prime \eta R^{-1})} VV @V{(\Id_A, ~\Id_A)}VV \\
				{\bf 1} @>>> (K,K,\alpha,\Id_K) @>(i^\prime_1, i^\prime_2)>>(H^\prime,G^\prime, \phi^\prime, R^\prime) @>{(\pi^\prime_1, \pi^\prime_2) }>> (A^{(\cdot)},A^{(\circ_T)}, \beta \; T, \Id_A) @>>> {\bf 1},
			\end{CD}
		\end{align}
which implies that $\mathcal{E}$ and $\mathcal{E}'$ are equivalent. Hence, the map $\Pi$ is injective. To see that $\Pi$ is surjective, let $$\mathcal{E}: \quad {\bf 1} \longrightarrow K_S \stackrel{i}{\longrightarrow}  H \stackrel{\pi}{\longrightarrow} A_T \longrightarrow {\bf 1}$$
be an extension of skew left braces representing an element in $ \Ext_{(\xi, \zeta, \epsilon)}(A_T, K_S)$. Consider the extension 
$$\mathcal{E}_{RRB} : \quad  {\bf 1} \longrightarrow (K, K, \alpha , \Id_K) \stackrel{(i, i)}{\longrightarrow}  (H^{(\cdot)}, H^{(\circ)}, \lambda^H, \Id_H) \stackrel{(\pi, \pi)}{\longrightarrow} (A^{(\cdot)}, A^{(\circ_T)}, \beta \, T, \Id_A) \longrightarrow {\bf 1}$$ of relative Rota-Baxter groups. In view of Proposition \ref{actnprop}, the action corresponding to the extension $\mathcal{E}_{RRB}$ is $(\nu, \mu, \sigma, f)$, and hence $\mathcal{E}_{RRB}$ represents an element in $ \Ext_{(\nu, \mu, \sigma, f)}(\mathcal{A}, \mathcal{K})$. Since the skew left brace extension induced by $\mathcal{E}_{RRB}$ is identical to $\mathcal{E}$, it follows that the map $\Pi$ is surjective. 
\end{proof}

\begin{cor}\label{isomorphism RRB and SLB cohomology}
Let $\mathcal{A}= (A,B, \beta, T)$ be a relative Rota-Baxter group and $\mathcal{K}=  (K,L,\alpha,S )$ a module over $\mathcal{A}$ with respect to the action $(\nu, \mu,\sigma, f)$. If $T$ and $S$ are bijections and $K$ and $L$ are abelian, then $\Lambda \Pi \Psi: \Ho^2_{RRB}(\mathcal{A}, \mathcal{K}) \to \Ho^2_N(A_R, K_S)$ is an isomorphism of groups.
\end{cor}

\begin{remark}
	The map $\Pi$ defined in Corollary \ref{RRB2sbext} need not be surjective in general. For example, let $\mathcal{K} = \mathcal{A} = (\mathbb{Z}_p, 1, \alpha, T)$ be two relative Rota-Baxter groups, where $p$ is a prime. Then every extension of $\mathcal{K}$ by $\mathcal{A}$ has the form
	$$\mathcal{E}: \quad {\bf 1} \longrightarrow (\mathbb{Z}_p, 1, \alpha, T) \longrightarrow (E, 1, \phi, R) \longrightarrow (\mathbb{Z}_p, 1, \alpha, T) \rightarrow {\bf 1}.$$
	Hence, the skew left brace extension induced by $\mathcal{E}$ is
	$$\mathcal{E}_{SB}: \quad {\bf 1} \longrightarrow \mathbb{Z}_p \longrightarrow E \longrightarrow \mathbb{Z}_p \longrightarrow {\bf 1},$$
	with $E$ being a trivial brace. It is easy to see that the associated action of any extension of $\mathcal{K}$ by $\mathcal{A}$ is trivial. Therefore, the associated action of the skew left brace extension induced by an extension of $\mathcal{K}$ by $\mathcal{A}$ is also trivial. Let $(H, \cdot, \circ)$ be the skew left brace with the additive group as the cyclic group $\mathbb{Z}_{p^2}$, and the multiplicative group given by $$x_1 \circ x_2 = x_1 + x_1 + p x_1 x_2$$ for all $x_1, x_2 \in H$. We know from \cite[p.4]{DB18} that the annihilator of $H$ is of order  $p$. Thus, we have the extension
	$$\mathcal{E}_{1}: \quad {\bf 1} \longrightarrow \mathbb{Z}_p \longrightarrow H \longrightarrow \mathbb{Z}_p \longrightarrow {\bf 1}$$  
	with the trivial associated action. Since $H$ is a non-trivial left brace, it follows that $[\mathcal{E}_1]$ cannot belong to the image of $\Pi$. Thus, $\Pi$ is not surjective, and consequently $\Lambda \Pi \Psi$ is not surjective in general.
\end{remark}

\section{Central and split extensions of relative Rota-Baxter groups} \label{sec central and split RRB}
Let $\mathcal{A} = (A, B, \beta, T)$ be an arbitrary relative Rota-Baxter group and  $\mathcal{K} = (K, L, \alpha, S)$ a trivial relative Rota-Baxter group, where $K$ and $L$ are abelian groups.  Let $\CExt(\mathcal{A}, \mathcal{K})$ denote the set of equivalence classes of all central extensions of $\mathcal{A}$ by $\mathcal{K}$.

Consider an extension in $\CExt(\mathcal{A}, \mathcal{K})$ and let $(\mu, \nu, \sigma, f)$ be its associated action. It is a direct check that $\mu$, $\nu$ and $\sigma$ are trivial homomorphisms. Furthermore, $f(l, a) = 1$ for all $l \in L$ and $a \in A$.  In other words, the associated action is trivial for a central extension. 
\par

For central extensions, the coboundary maps can be simplified. More precisely,  we define $\partial_{CRRB}^1: C^1_{RRB} \longrightarrow C^{2}_{RRB}$ by 
$\partial_{CRRB}^1(\theta_1, \theta_2)=(\partial^1(\theta_1), \partial^1(\theta_2), \lambda_1, \lambda_2  )$, where $\lambda_1$ and $\lambda_2$ are given by
\begin{eqnarray*}
\lambda_1(a,b) &=& \theta_1(a) \,(\theta_1(\beta_b(a)))^{-1},\\
\lambda_2(a)  &=& S(\theta_1(a)) \,(\theta_2(T(a)))^{-1}.
\end{eqnarray*}
Similarly, we define $\partial_{CRRB}^2: C^2_{RRB} \longrightarrow C^{3}_{RRB}$  by  $\partial^2_{RRB}(\tau_1, \tau_2, \rho, \chi)=(\partial^2(\tau_1), \partial^2(\tau_2), \gamma_1, \gamma_2, \gamma_3)$, where
\begin{eqnarray*}
\gamma_1(a,b_1, b_2) &=& \rho(a, b_1 b_2)  \,  (\rho(\beta_{b_2}(a), b_1))^{-1} \,(\rho(a,b_2))^{-1},\\
\gamma_2(a_1, a_2,b) &=& \rho(a_1 a_2, b) \,(\rho(a_1,b))^{-1}(\rho(a_2, b))^{-1} \tau_1(a_1, a_2) \, (\tau_1(\beta_{b}(a_1), \beta_{b}(a_2)))^{-1}, \\
\gamma_3(a_1, a_2) &=&  S\big(\rho(a_2, T(a_1)) \,\tau_1(a_1,\beta_{T(a_1)}(a_2)) \big) \, (\tau_2(T(a_1), T(a_2)))^{-1} \,(\delta^1(\chi)(a_1, a_2))^{-1}.
\end{eqnarray*}

By Lemma \ref{rrb coboundary condition}, we have $\IM(\partial^1_{CRRB}) \subseteq \ker(\partial^2_{CRRB})$, and hence we can  define  $$\Ho^2_{CRRB}(\mathcal{A}, \mathcal{K})=\Ker(\partial^2_{CRRB})/  \IM(\partial^1_{CRRB}).$$

The proof of Theorem \ref{ext and cohom bijection} yields the following result.

\begin{thm}
There is a bijection between $\CExt(\mathcal{A}, \mathcal{K})$ and $\Ho^2_{CRRB}(\mathcal{A}, \mathcal{K})$.
\end{thm}
\medskip

Next, we classify a specific class of extensions of relative Rota-Baxter groups, which may not be abelian in general.

\begin{defn}
Suppose that
$$\mathcal{E} : \quad {\bf 1} \longrightarrow (K,L, \alpha,S ) \stackrel{(i_1, i_2)}{\longrightarrow}  (H,G, \phi, R) \stackrel{(\pi_1, \pi_2)}{\longrightarrow} (A,B, \beta, T) \longrightarrow {\bf 1}$$ is an extension of relative Rota-Baxter groups. We say that $\mathcal{E}$ is a split-extension if there exists a set-theoretic section $(s_H, s_G)$ to  $\mathcal{E}$ which is also a morphism of relative Rota-Baxter groups.
\end{defn}

\begin{remark}
If $\mathcal{E}$ is a split extension of relative Rota-Baxter groups, then it follows that the extensions induced for both groups and skew left braces also split within their respective categories. Further, the classification of split extensions of skew left braces can be found in \cite{N1}.
\end{remark}

Suppose that $$\mathcal{E} : \quad {\bf 1} \longrightarrow (K,L, \alpha,S ) \stackrel{(i_1, i_2)}{\longrightarrow}  (H,G, \phi, R) \stackrel{(\pi_1, \pi_2)}{\longrightarrow} (A,B, \beta, T) \longrightarrow {\bf 1}$$ is a split-extension of relative Rota-Baxter groups. Then existence of a morphism $(s_H, s_G): (A, B, \beta, T) \rightarrow (H, G, \phi, R)$ implies that $s_G \; T = R \;  s_H$ and $s_H \; \beta_b = \phi_{s_G(b)} \; s_H$ for all $b \in B$. This means $\rho$ and $\chi$ defined in \eqref{atilde} and \eqref{RRB2}, respectively, are trivial maps. Moreover, by \eqref{feqn} and \eqref{relativecon}, the maps $\phi$ and $R$ are given by
\begin{eqnarray*}
\phi_{(s_G(b)l)}(s_H(a)k) &=&s_H(\beta_b(a)) ~\nu_b(f(l,a)\alpha_l(k)),\\
R(s_H(a)k) &=& s_H(T(a)) ~S(\nu^{-1}_{T(a)}(k)),
\end{eqnarray*}
for all $a \in A$, $b \in B$, $k \in K$ and $l \in L$. 
\par 
Recall that,  the group operations in the semi-direct products $A \times_{\mu} K$ and $B \ltimes_{\sigma} L$ are given by
\begin{eqnarray*}
	(a_1,k_1) (a_2, k_2) &=&(a_1 a_2, ~\mu_{a_2}(k_1)k_2),\\
	(b_1,l_1) (b_2, l_2) &= & (b_1 b_1, ~\sigma_{b_2}(l_1)l_2),
\end{eqnarray*}
for all $a_1, a_2 \in A$, $b_1, b_2, \in B$, $k_1, k_2 \in K$ and $l_1, l_2 \in L$. Routine calculations yield the following proposition.

\begin{prop}
Let $(K,L, \alpha,S )$  and  $(A,B, \beta, T)$ be two relative Rota-Baxter groups. Suppose that there exists a map $f: L \times A \rightarrow K$,  a homomorphism $\nu: B \rightarrow \Aut(K)$ and anti-homomorphisms $\mu: A \rightarrow \Aut(K)$ and $\sigma :B \rightarrow \Aut(L)$. Then the map $\phi: B \ltimes_{\sigma} L \rightarrow  \Aut(A \times_{\mu} K)$ given by $$ \phi_{(b,l)}(a,k)=\big(\beta_b(a),~ \nu_b(f(l,a)\alpha_l(k))\big),$$ for $(b, l) \in   B \ltimes_{\sigma} L$ and $(a,k) \in  A \times_{\mu} K$, is a homomorphism if and only the following conditions are satisfied:
	\begin{eqnarray}
		\nu_{b_2}\big(f(\sigma_{b_2}(l_1) l_2,a_1)  ~\alpha_{\sigma_{b_2}(l_1) l_2}(k_1)\big)& =& f(l_1, \beta_{b_2}(a_1)) ~\alpha_{l_2}\big(\nu_{b_2}(f(l_2, a_1) \alpha_{l_2}(k_1))\big),\\
		\nu_{b_1}\big(f(l_1, a_1 a_2) ~\alpha_l(\mu_{a_2}(k_1) k_2) \big) &=& \mu_{\beta_{b_1}(a_2)}\big(\nu_{b_1}(f(l,a_1) \alpha_l(k_1))\big) ~\nu_{b_1}(f(l,a_2) \alpha_l(k_2)),
	\end{eqnarray}
for all $a_1, a_2 \in A$, $b_1, b_2, \in B$, $k_1, k_2 \in K$ and $l_1, l_2 \in L$.

Furthermore, if $\phi$ is a group homomorphism, then the map $R: A \times_{\mu} K \rightarrow B \ltimes_{\sigma} L$ given by 
$$R(a, k)=(T(a), ~S(\nu^{-1}_{T(a)}(k))),$$ for $(a,k) \in  A \times_{\mu} K$, is a relative Rota-Baxter operator if and only if 
\begin{eqnarray*}
&& \sigma_{T(a_2)}(S(\nu^{-1}_{T(a_1)}(k_1))) \, S(\nu^{-1}_{T(a_2)}(k_2))\\
 &=& S\big(\nu^{-1}_{T(a_1) T(a_2)}\big(\mu_{\beta_{T(a_1)}(a_1)}(k_1) ~\nu_{T(a_1)}(f(S(\nu^{-1}_{T(a_1)}(k_1)), a_2) ~\alpha_{S(\nu^{-1}_{T(a_1)}(k_1))}(k_2)) \big) \big),
	\end{eqnarray*}
for all $a_1, a_2 \in A$ and $k_1, k_2 \in K$.
\end{prop}

The induced relative Rota-Baxter group $(A \times_{\mu} K, B \ltimes_{\sigma} L , \phi, R)$ obtained in the preceding proposition is called the \emph{semi-direct product} of $(A,B, \beta, T)$ by $(K,L, \alpha,S )$ with respect to the action $(\nu, \mu, \sigma, f)$. The direct product of relative Rota-Baxter groups is a particular case of the semi-direct product when all the maps $\nu, \mu,\sigma, f$ are trivial.

\begin{thm}\label{semi-direct product and split extensions}
The semi-direct product of relative Rota-Baxter groups $(A, B, \beta, T)$ by $(K, L, \alpha, S)$ gives rise to a split-extension of $(A, B, \beta, T)$ by $(K, L, \alpha, S)$. Furthermore, any split-extension of relative Rota-Baxter groups is equivalent to an extension defined by the semi-direct product of relative Rota-Baxter groups.
\end{thm}

\begin{remark}
Note that abelian split extensions of relative Rota-Baxter groups correspond to the zero element of the second cohomology. Further, the skew left brace induced from the semi-direct product  of $(A, B, \beta, T)$ by  $(K, L, \alpha, S)$ is the semi-direct product of skew left braces $A_{T}$ by $K_{S}$ as defined in \cite{N1}. 
\end{remark}
\medskip
 	
\begin{ack}
{\rm Pragya Belwal thanks UGC for the PhD research fellowship and Nishant Rathee thanks IISER Mohali for the institute post doctoral fellowship. Mahender Singh is supported by the SwarnaJayanti Fellowship grants DST/SJF/MSA-02/2018-19 and SB/SJF/2019-20.}
\end{ack}

\section{Declaration}
The authors declare that they have no conflict of interest.

\end{document}